\documentclass{amsart}
\usepackage{mathtools}
\usepackage{amsfonts}
\usepackage{amssymb}
\usepackage{palatino}
\usepackage{tikz}
\usetikzlibrary{ decorations.markings}
\usepackage{float}

\DeclareMathOperator{\ad}{ad}

\DeclareFontFamily{U}{mathx}{}
\DeclareFontShape{U}{mathx}{m}{n}{<-> mathx10}{}
\DeclareSymbolFont{mathx}{U}{mathx}{m}{n}
\DeclareMathAccent{\widehat}{0}{mathx}{"70}
\DeclareMathAccent{\widecheck}{0}{mathx}{"71}

\theoremstyle{plain}

\newtheorem{theorem}{Theorem}[section]
\newtheorem{corollary}[theorem]{Corollary}
\newtheorem{lemma}[theorem]{Lemma}
\newtheorem{proposition}[theorem]{Proposition}
\newtheorem{problem}[theorem]{Problem}
\theoremstyle{remark}
\newtheorem{remark}[theorem]{Remark}

\numberwithin{equation}{section}

\definecolor{deepgreen}{cmyk}{1,0,1,0.5}

\usepackage{hyperref}
\usepackage{mathtools}

\newcommand{\eps}{\varepsilon}

\newcommand{\norm}[2]{\left\Vert #1\right\Vert_{#2}}

\newcommand{\ba}{\breve{a}}
\newcommand{\bb}{\breve{b}}
\newcommand{\bbC}{\mathbb{C}}
\newcommand{\bbR}{\mathbb{R}}

\newcommand{\dbar}{\overline{\partial}}

\newcommand{\zbar}{\overline{z}}
\newcommand{\twomat}[4]
{
\left(
	\begin{array}{cc}
		{#1}	&	{#2}	\\[10pt]
		{#3}	&	{#4}
		\end{array}
\right)
}

\newcommand{\Twomat}[4]
{
\left[
	\begin{array}{cc}
		{#1}	&	{#2}	\\[10pt]
		{#3}	&	{#4}
		\end{array}
\right]
}

\newcommand{\diagmat}[2]
{
\left(
	\begin{array}{cc}
		{#1}	&	0	\\
		0		&	{#2}
		\end{array}
\right)
}
\newcommand{\offdiagmat}[2]
{
\left(
	\begin{array}{cc}
		0			&		{#1} 	\\
		{#2}		&		0
		\end{array}
\right)
}
\newcommand{\SixMatrix}[6]
{
\begin{figure}
\centering
\caption{#1}
\vskip 15pt
\begin{tikzpicture}
[scale=0.75]
\draw[dashed, thick]	 (-6,0) -- (6,0);
\draw[thick] 	(-4,4) -- (4,-4);
\draw[thick] 	(-4,-4) -- (4,4);
\draw	[fill]		(0,0)						circle[radius=0.075];
\node[below] at (0,-0.1) 				{$z_0$};
\node[above] at (3.5,2.5)				{$\Omega_1$};
\node[below]  at (3.5,-2.5)			{$\Omega_6$};
\node[above] at (0,3.25)				{$\Omega_2$};
\node[below] at (0,-3.25)				{$\Omega_5$};
\node[above] at (-3.5,2.5)			{$\Omega_3$};
\node[below] at (-3.5,-2.5)			{$\Omega_4$};
\node[above] at (0,1.25)				{$\twomat{1}{0}{0}{1}$};
\node[below] at (0,-1.25)				{$\twomat{1}{0}{0}{1}$};
\node[right] at (1.40,0.90)			{$#3$};
\node[left]   at (-1.40,0.90)			{$#4$};
\node[left]   at (-1.40,-0.90)			{$#5$};
\node[right] at (1.40,-0.90)			{$#6$};
\end{tikzpicture}
\label{fig 1}
\end{figure}
}
\begin{document}
	\title{perturbation of the nonlinear Schr\"odinger equation
 by a localized nonlinearity}
	\author{Gong Chen}
	\author{Jiaqi Liu}
	\author{Yuanhong Tian}
\address[Chen]{School of Mathematics, Georgia institute of technology. 686 Cherry Street, Skiles Building, Atlanta GA 30332-0160}
\email{gc@math.gatech.edu}
\address[Liu]{School of mathematics, University of Chinese Academy of Sciences. No.19 Yuquan Road, Beijing China }
\email{jqliu@ucas.ac.cn}
\address[Tian]{School of mathematics, University of Chinese Academy of Sciences. No.19 Yuquan Road, Beijing China}
\email{1838393768@qq.com}	
\thanks{The first author was partially supported by NSF grant DMS-2350301 and by Simons foundation MP-TSM-00002258. The Second author is partially supported by NSFC grant No.12201605 and CAS grant E4ER0101A2.}
\begin{abstract}
We revisit the perturbative theory of infinite dimensional integrable systems developed by P. Deift and X. Zhou \cite{DZ-2}, aiming to provide  new and simpler proofs of some  key $L^\infty$ bounds and  $L^p$ \emph{\textit{a priori}} estimates. Our proofs emphasizes a further step towards understanding focussing problems and extends the applicability  to other integrable models. As a concrete application, we examine the perturbation of the one-dimensional defocussing cubic nonlinear Schr\"odinger equation by a localized higher-order term. We introduce improved estimates to control the power of the perturbative term and demonstrate that the perturbed equation exhibits the same long-time behavior as the completely integrable nonlinear Schr\"odinger equation.
    %  We revisit the perturbative theory of intergrable systems developed by P. Deift and X. Zhou \cite{DZ-2}. Our goal is to establish new and simpler proof of  certain \emph{\textit{a priori}} estimates from the perspective towards understanding focusing problems and applicable to other integrable modes. As a concrete application,  we apply our estimates to  consider the perturbation of the 1-d defocussing cubic nonlinear Schr\"odinger equation by a localized higher order term. We make improvements on estimates 
   %   to control the power of perturbative term.  We show that the perturbed equation exhibits the same long time behavior as the completely integrable nonlinear Schr\"odinger equation.

    %  We revisit the perturbative theory of intergrable systems developed by P. Deift and X. Zhou. Our goal is to  make improvements on certain \textit{a priori} estimates
     % to control the power of perturbative term. We consider the perturbation of the 1-d defocussing cubic nonlinear Schr\"odinger equation by a localized higher order term. We show that the perturbed equation exhibits the same long time behavior as the completely integrable nonlinear Schr\"odinger equation. 
  
\end{abstract}

\maketitle

\section{Introduction}

\subsection{Background}
It is well-known that the $1$-d defocussing cubic nonlinear Schr\"odinger equation (NLS)
\begin{equation}
\label{eq:nls}
    iq_t+q_{xx}-2|q|^2q=0
\end{equation}
is a canonical example of  infinite dimensional integrable systems,  in the sense that it is the compatibility condition for the following commutator:
\begin{align}
\label{eq: compat}
    \left[\partial_x-U, \partial_t-W\right]=0,
\end{align}
where
\begin{align*}
U & =i z \sigma+\left(\begin{array}{cc}
0 & q \\
\bar{q} & 0
\end{array}\right), \\
W & =-i z^2 \sigma-z\left(\begin{array}{cc}
0 & q \\
\bar{q} & 0
\end{array}\right)+\left(\begin{array}{cc}
-i|q|^2 & i \partial_x q \\
-i \partial_x \bar{q} & i|q|^2
\end{array}\right).
\end{align*}
Moving beyond integrable structures, in the seminal paper by P.~Deift and X.~Zhou \cite{DZ-2}, after exploring the $L^p$ \emph{a priori} estimates of the resolvent operator associated to the Riemann-Hilbert problem, the authors investigated the perturbation of the cubic NLS with higher order nonlinearity:
\begin{equation}
    iq_t+q_{xx}-2|q|^2q-\eps|q|^\ell q=0.
\end{equation}
with $\ell>7/2$. 
%By establishing the existence of wave operator (\cite[(1.17)]{DZ-2}) 
%\begin{equation}
%   W^{+}(q) \equiv \lim _{t \rightarrow \infty} U_{-t}^{\mathrm{NLS}} \circ U_t^{\varepsilon}(q) 
%\end{equation}
The authors show that for $\varepsilon\ll 1$ whose size depends on the $H^{1,1}$ Sobolev norm of the initial data, solutions to perturbed NLS equation behave as $t\to \infty$ like solutions of the cubic NLS equation \eqref{eq:nls} which is completely integrable. Their work is the first major step towards building a rigorous perturbation theory for infinite dimensional integrable systems. \cite{DZ-2} relies partly on the $L^p$ \textit{a priori} estimate developed in \cite{DZ-1}. As a byproduct, Deift and Zhou calculated the long time asymptotics for Equation \eqref{eq:nls} whose initial data belongs to the $H^{1,1}$ space. Later, in \cite{DMM18} (see also a published version of this paper \cite{DMM}), Dieng and McLaughlin re-investigated the long time asymptotics of \eqref{eq:nls} and developed the $\overline{\partial}$ version of the nonlinear steepest descent method which is originally developed in \cite{DZ93}. We also mention that in \cite{BI}, the authors investigate the evolution of scattering data under perturbations of the Toda lattice.
%In the current paper, we revisit the perturbative method of \cite{DZ-2} to calculate the long time asymptotics of \eqref{eq: pnls} aiming to obtain several new \textit{a priori} estimates that are more adaptable to other completely integrable equations, in the case of both defocussing and focussing.

In this paper, we revisit the perturbative method of \cite{DZ-2} and  give  new  proofs of $L^p$  \emph{\textit{a priori}} estimates and also $L^\infty$ bounds aiming at a further step towards understanding perturbation of focussing problems and the applicability of \cite{DZ-2} to other
integrable models. To illustrate applications of improved estimates by our new approach, we consider the long time dynamics for the following 1-d defocussing NLS with a localized nonlinearity:
\begin{align}
\label{eq: pnls}
    &iq_t+q_{xx}-2|q|^2q-\eps a(x)|q|^l q=0,\\
    \nonumber
    & q(x, t=0)=q_0(x) \in H^{1,1}(\mathbb{R}).
\end{align}
Here $0<\varepsilon\ll 1$, $l>3$ and $a(x)$ is a Schwarz function (We do not aim to pursue the optimal conditions for 
 $a(x)$ here). 
If we replace the RHS of \eqref{eq: compat} with $\varepsilon G(x)$ where
\begin{equation}
\label{G}
  G(x) =-ia(x)|q|^l\left(\begin{array}{cc}
0 & q \\
-\bar{q} & 0
\end{array}\right) 
\end{equation}
then it is easy to check that 
\begin{equation}
    \label{eq: pcompat}
    \left[\partial_x-U, \partial_t-W\right]=\varepsilon G(x)
\end{equation}
is the compatibility condition for \eqref{eq: pnls}.

In the following, we will set up problems and introduce notations based on the concrete example \eqref{eq: pnls}.

\subsection{The inverse scattering for the NLS}
Recall that \eqref{eq:nls} generates an isospectral flow for the problem
\begin{equation}
\label{L}
\frac{d}{dx} \Psi = iz \sigma \Psi + Q(x) \Psi
\end{equation}
where
$$ \sigma= \diagmat{\frac{1}{2}}{-\frac{1}{2}}, \,\,\, Q(x) = \offdiagmat{q(x)}{\overline{q(x)}}.$$
This is a standard AKNS system. If $u \in L^1(\bbR) $, Equation \eqref{L} admits bounded 
solutions for $z \in \mathbb{R}$.   There exist unique solutions $\Psi^\pm$ of \eqref{L} obeying the the following space asymptotic conditions
$$\lim_{x \to \pm \infty} \Psi^\pm(x,z) e^{-ix z \sigma_3} = \diagmat{1}{1},$$
and there is a matrix $S(z)$, the scattering matrix, with 
 $\Psi^+(x,z)=\Psi^-(x,z) S(z)$.
The matrix $T(z)$ takes the form
\begin{equation} \label{matrixT}
 S(z) = \twomat{a(z)}{\bb(z)}{b(z)}{\ba(z)} 
 \end{equation}
and  the determinant relation gives
$$ a(z)\ba(z) - b(z)\bb(z) = 1 $$
Combining this with the symmetry relations 
\begin{align} \label{symmetry}
\ba(z)=\overline{a( \zbar )}, \quad \bb(z) = \overline{ b(\zbar)}. 
\end{align}
we arrive at
$$|a(z)|^2-|b(z)|^2=1$$
and conclude that $a(z)$ is zero-free. By the standard inverse scattering theory, we formulate the reflection coefficient:
\begin{equation}
\label{reflection}
r(z)=\bb(z)/a(z), \quad z\in\bbR.
\end{equation}
In \cite{Zhou98}, it is shown that for $k, j$ integers with $k\geq 0$, $j\geq 1$, the direct scattering map $\mathcal{R}$ maps $H^{k,j}(\bbR)$ onto $H^{j,k}_1=H^{j,k}(\bbR)\cap \lbrace r: \norm{r}{L^\infty} <1\rbrace$ where $H^{j,k}$ norm is given by:
\begin{equation}
\label{sp: weighted}
    \norm{q}{H^{i,j}(\bbR)}
= \left( \norm{(1+|x|^j)q}{2}^2 + \norm{q^{(i)}}{2}^2 \right)^{1/2}. 
\end{equation}
and the map $\mathcal{R}: q\mapsto r$ is Lipschitz continuous. Since we are dealing with the defocussing NLS, only the reflection coefficient $r$ is needed for the reconstruction of the solution. 
The long-time behavior of the solution to the NLS equation is obtained through a sequence of transformations of the following RHP:
\begin{problem}
\label{prob:NLS.RH0}
Given $r(z) \in H^{1,1}(\bbR)$, 
find a $2\times 2$ matrix-valued function $m(z;x,t)$ on $\bbC \setminus \bbR$ with the following properties:
\begin{enumerate}
\item		$m(z;x,t) \to I$ as $|z| \to \infty$,
			\medskip
			
\item		$m(z;,x,t)$ is analytic for $z \in \mathbb{C} \setminus \bbR$ with continuous boundary values
			$$m_\pm(z;x,t) = \lim_{\varepsilon \to 0} m(z\pm i\varepsilon;x,t),$$
			\medskip
			
\item		
The jump relation $m_+(z;x,t) = m_	-(z;x,t) e^{i\theta \ad \sigma} v(z)$ holds, where
			\begin{equation}
			\label{NLS.V}
			e^{i\theta \ad \sigma} v(z)=\twomat{1-|r(z)|^2}
			{r(z)e^{i\theta}}
   {-\overline{r(z)}e^{-i\theta}}
   {1}
			\end{equation}
and the real phase function $\theta$ is given by
\begin{equation}
\label{NLS.phase}
\theta(z;x,t) = xz-tz^2
\end{equation}
with stationary points
\begin{equation}
    \label{stationary pt}
    z_0={ \dfrac{x}{2t}} .
    \end{equation}
    \end{enumerate}
\end{problem}
From the solution of Problem \ref{prob:NLS.RH0}, we recover
\begin{align}
\label{nls.q}
q(x,t) &= \lim_{z \to \infty} 2 i z M_{12}(x,t,z)
\end{align}
where the limit is taken in $\bbC\setminus \bbR$ along any direction not tangent to $\bbR$.  
\subsubsection{The Beals-Coifman solution} For $z\in \mathbb{R}$, introduce
\begin{equation}
    w_\theta=\left(w^-_\theta, w^+_\theta\right) = \left(\twomat{0}{r(z)e^{i\theta}}{0}{0}, \twomat{0}{0}{-\overline{r(z)}e^{-i\theta}}{0}\right)
    \end{equation}
    corresponding to the factorization of the jump matrix:
    \begin{equation}
        v_\theta=\left(v^{-}_\theta\right)^{-1} v^{+}_\theta= \twomat{1}{-r(z)e^{i\theta}}{0}{1}^{-1} \twomat{1}{0}{-\overline{r(z)}e^{-i\theta}}{1}.
    \end{equation}
Denote the associated singular integral operator,
\begin{equation}
    C_{v_\theta} h=C^{+}\left(h w^{-}_\theta\right)+C^{-}\left(h w^{+}_\theta\right)
\end{equation}
and let $\mu\in I+L^2(\bbR)$ solves the equation
\begin{equation}
\label{eq:mu}
     \left(1-C_{v_\theta} \right)^{-1} \mu=I.
\end{equation}
Then a simple calculation shows that
\begin{equation}
    m_{ \pm}=I+C^{ \pm}\left(\mu\left(w^{+}_\theta+w^{-}_\theta\right)\right)=\mu v^\pm_\theta
\end{equation}
is the unique solution to Problem \ref{prob:NLS.RH0}. Set
\begin{equation}\label{eq:boldQ}
    \mathbf{Q}(x,t)=\int_{\mathbb{R}} \mu(x, s)\left(w_\theta^{+}(s)+w_\theta^{-}(s)\right) d s.
\end{equation}
Then direct calculations give:
\begin{equation}
    \twomat{0}{q(x,t)}{\overline{q(x,t)}}{0}=\frac{1}{2 \pi} \mathrm{ad} \sigma(\mathbf{Q}).
\end{equation}
We now state the important result from \cite{DZ-3}:
\begin{theorem}
    Given $q_0$ the initial data of Equation \eqref{eq:nls}, then we have the direct scattering map
    \begin{equation}
        \mathcal{R}: H^{1,1}(\bbR)\ni q_0\mapsto r(z)\in H^{1,1}_1(\bbR)
    \end{equation}
    and inverse scattering map
     \begin{equation}
        \mathcal{R}^{-1}:  H^{1,1}_1(\bbR)\ni e^{i\theta}r(z)\mapsto q(x,t)\in H^{1,1}(\bbR).
    \end{equation}
    Both $\mathcal{R}$ and $\mathcal{R}^{-1}$ are Lipschitz continuous. Moreover, $ \mathcal{R}(q(t))$ evolves linearly
    \begin{equation}
    \label{time-1}
        \mathcal{R}(q(t))=\mathcal{R}(q(0)) e^{-i t z^2}=r(z) e^{-i t z^2}.
    \end{equation}
\end{theorem}
\subsection{The inverse scattering for the perturbed NLS} Let $q(x,t)$ be the solution to Equation \eqref{eq: pnls}. We then let 
\begin{equation}
    Q(x,t)=\twomat{0}{q(x,t)}{\overline{q(x,t)}}{0}
\end{equation}
be the potential of \eqref{L}. Following the standard direct and inverse scattering formalism, we can as well formulate the following Riemann-Hilbert problem:
\begin{problem}
\label{prob:NLS.RH1}
Find a $2\times 2$ matrix-valued function $m(z;x,t)$ on $\bbC \setminus \bbR$ with the following properties:
\begin{enumerate}
\item		$m(z;x,t) \to I$ as $|z| \to \infty$,
			\medskip
			
\item		$m(z;,x,t)$ is analytic for $z \in \mathbb{C} \setminus \bbR$ with continuous boundary values
			$$m_\pm(z;x,t) = \lim_{\varepsilon \to 0} m(z\pm i\varepsilon;x,t),$$
			\medskip
			
\item		
The jump relation $m_+(z;x,t) = m_	-(z;x,t) e^{ix \ad \sigma} v(z;q(t))$ holds, where
			\begin{equation}
			\label{NLS.V}
			e^{ix\ad \sigma} v(z;q(t))=\twomat{1-|r(z;q(t))|^2}
			{r(z;q(t))e^{ix}}
   {-\overline{r(z;q(t))}e^{-ix}}
   {1}.
			\end{equation}
\end{enumerate}
\end{problem}
Define
\begin{equation}
    r(t)(z)=e^{i t z^2} r(z ; q(t))
\end{equation}
and the following integral equation is originally obtained in \cite{Kaup} and \cite{KN}:
\begin{equation}
\label{eq:r}
    r(t)(z)=r_0(z)+\varepsilon \int_0^t d s e^{i z^2 s} \int_{-\infty}^{\infty} d y e^{-i y z}\left(m_{-}^{-1}(z ; y,q(s)) G(q(y, s)) m_{-}(z ; y, q(s))\right)_{12}
\end{equation}
where
\begin{equation}
    r_0(z)=\mathcal{R}\left(q_0\right)(z), \quad q_0=q(t=0).
\end{equation}
%\Red{Please check notations here. the last variable of $m_\pm$ is different from it is deifinition above \eqref{NLS.V}}

\subsection{Main results}
Now we introduce our main results for the asymptotics of \eqref{eq: pnls}. Due to their technicality, we postpone the presentation of our improved  estimates in Section \ref{sec:dbar} and Section \ref{sec:priori}. 

\begin{theorem}
\label{thm:r}
Assume $q_0=q(x,0)\in H^{1,1} $  with $\norm{r_0}{H^{1,1}}=\eta$ and $\norm{r_0}{L^\infty}=\rho<1$. There exists $\varepsilon(\rho, \eta)\ll 1$ such that $q(x,t) \in C\left([0, \infty), H^{1,1}\right)$ solves \eqref{eq: pnls} with  $
\varepsilon=\varepsilon(\rho, \eta)$ and $r(t)(z)=e^{i t z^2} \mathcal{R}(q(t))(z)$ uniquely solves \eqref{eq:r}. Moreover, we have that for $t\in (0,+\infty)$, $\norm{r(t)}{H^{1,1}}<2\eta$ and $\norm{r(t)}{L^\infty}<(1+\rho)/2$ .
\end{theorem}
\begin{proof}
    Equation \eqref{eq:r} can be written in the following form:
\begin{equation}
    \label{eq:rF}
    r(t)(z)=r_0+\varepsilon \int_0^t F(s, r(s)) d s
\end{equation}
with
\begin{equation}
\label{F}
    F(z,t;r)=P_{12} \int_{-\infty}^{\infty} e^{-i\left(y z-t z^2\right)} m_{-}^{-1}\left(y, z ; \mathcal{R}^{-1}\left(e^{-i z^2 t} r\right)\right) G\left(\mathcal{R}^{-1}\left(e^{-i z^2 t} r\right)\right)(y) m_{-}\left(y, z ; \mathcal{R}^{-1}\left(e^{-i z^2 t} r\right)\right) d y
\end{equation}
%\Red{Please check the last variable here to see if $\mathcal{R}^{-1}$ is missing.}
where $P_{12}$ denotes the projection of $(2\times2)$-matrices onto their $(1, 2)$-entries. For each fixed $t>0$ and $0<p-2\ll 1$, picking
\begin{equation}
\label{epsi1}
    \varepsilon <\frac{l/2-3/2}{ C_{F}^{1,1}(\rho, \eta)}
\end{equation}
and 
\begin{equation}
\label{epsi2}
    \varepsilon <\frac{l/2+1/2p-7/4}{C_{\Delta F}^{1,1}(\rho, \eta) }
\end{equation}where $C_{\Delta F}^{1,1}(\rho, \eta)$ and $ C_{F}^{1,1}(\rho, \eta)$ are from Proposition \ref{prop:G-estimate}, the conclusion of Proposition \ref{prop:G-estimate}
implies that \eqref{eq:rF} is a contraction, whence,  existence and uniqueness of the solution follow. Moreover, it is easy to check that
\begin{equation}
\label{epsi3}
    \varepsilon <\frac{\eta}{2}\frac{l/2-3/2}{C_{F}^{1,1}((\rho+1)/2, 2\eta)} ,
\end{equation}
and 
\begin{equation}
\label{epsi4}
     \varepsilon <\frac{1-\rho}{3}\frac{l/2-3/2}{ C_{F}^{1,1}((\rho+1)/2, 2\eta)}
\end{equation}
imply that $\norm{r(t)}{H^{1,1}}<2\eta$ and $\norm{r(t)}{L^\infty}<(1+\rho)/2$. Therefore we choose $\varepsilon=\varepsilon(\rho, \eta) $ satisfying \eqref{epsi1}-\eqref{epsi4} and then the desired results follow.
\end{proof}
\begin{theorem}
\label{main}
  For $q_0=q(x,0)\in H^{1,1}$, if  $q(x,t)$ solves Equation \eqref{eq: pnls} with $\varepsilon=\varepsilon(\rho, \eta)$ satisfying \eqref{epsi1}-\eqref{epsi4}, then $q(x,t)$ admits the following asymptotic expansion:
  \begin{equation}
      q(x,t)=q_{as}(x,t)+\mathcal{O}(t^{-3/4})
  \end{equation}
  where
  \begin{align*}
      q_{as}(x,t)&=t^{-1 / 2} \alpha\left(z_0\right) e^{i x^2 / 4 t-i \nu\left(z_0\right) \log 2 t},\\
      \nu\left(z_0\right) &=-\frac{1}{2 \pi} \log \left(1-\left|r(t)\left(z_0\right)\right|^2\right), \quad \left|\alpha\left(z_0\right)\right|^2=\frac{1}{2} \nu\left(z_0\right)\\
      \arg \alpha\left(z_0\right)&=\frac{1}{\pi} \int_{-\infty}^{z_0} \log \left(z_0-z\right) d\left(\log \left(1-\left|r(t)(z)\right|^2\right)\right)+\frac{1}{4} \pi+\arg \Gamma\left(i \nu\left(z_0\right)\right)+\arg \left[r(t)\left( z_0\right) \right].
      \end{align*}
      Here $\Gamma $ is the gamma-function, $z_0=x / 2 t$,  and the error term is uniform in $x$.
\end{theorem}

\begin{proof}
    The asymptotic formula follows from Theorem \ref{thm:r} and an application of the nonlinear steepest descent method as is given by
    \cite{DMM}.
\end{proof}
From Theorem \ref{thm:r} we deduce that for $t_2>t_1\gg 0$,
\begin{equation}
    \left\|r\left(t_2\right)-r\left(t_1\right)\right\|_{H^{1,1}} \leqslant \frac{ C_{F}^{1,1}((\rho+1)/2, 2\eta)}{l/2-1/2}\left(\frac{1}{\left(1+t_1\right)^{l/2-3/2}}-\frac{1}{\left(1+t_2\right)^{l/2-3/2}}\right),
\end{equation}
and so $\lbrace r(t) \rbrace$ is a Cauchy sequence in $t$ and $r_\infty:=\lim _{t \rightarrow \infty} r(t)$ exists in $H^{1,1}$. Then by \cite[(4.21)]{DZ-2}, with $\mathbf{Q}$ from \eqref{eq:boldQ},
\begin{equation}
    \left\|\mathbf{Q}\left(e^{-i \delta^2 t} r(t)\right)-\mathbf{Q}\left(e^{-i \delta^2 t} r_{\infty}\right)\right\|_{L^{\infty}(d x)} \leqslant \frac{c}{(1+t)^{1 / 2 p+1 / 4}}\left\|r(t)-r_{\infty}\right\|_{H^{1,1}},
\end{equation}
by choosing suitable $l$ and $p$, we arrive at the following corollary:
\begin{corollary}
 For $q_0=q(x,0)\in H^{1,1}$, if  $q(x,t)$ solves Equation \eqref{eq: pnls} with $\varepsilon=\varepsilon(\rho, \eta)$ satisfying \eqref{epsi1}-\eqref{epsi4}, then for some $\kappa>0$, $q(x,t)$ admits the following asymptotic expansion:
  \begin{equation}
      q(x,t)=\tilde{q}_{as}(x,t)+\mathcal{O}\left(t^{-1/2-\kappa}\right)
  \end{equation}
  where
  \begin{align*}
      \tilde{q}_{as}(x,t)&=t^{-1 / 2} \alpha\left(z_0\right) e^{i x^2 / 4 t-i \nu\left(z_0\right) \log 2 t},\\
      \nu\left(z_0\right) &=-\frac{1}{2 \pi} \log \left(1-\left|r_\infty\left(z_0\right)\right|^2\right), \quad \left|\alpha\left(z_0\right)\right|^2=\frac{1}{2} \nu\left(z_0\right)\\
      \arg \alpha\left(z_0\right)&=\frac{1}{\pi} \int_{-\infty}^{z_0} \log \left(z_0-z\right) d\left(\log \left(1-\left|r_\infty(z)\right|^2\right)\right)+\frac{1}{4} \pi+\arg \Gamma\left(i \nu\left(z_0\right)\right)+\arg r_\infty\left( z_0\right).
      \end{align*}
      Here $\Gamma $ is the gamma-function, $z_0=x / 2 t$,  and the error term is uniform in $x$.
\end{corollary}
\subsection {Outline of the paper}
The paper is organized as follows:
\begin{itemize}
    \item[1.]In Section \ref{sec:dbar},  we present a new and simpler proof to the $L^p$ \textit{a priori} bound of the resolvent operator  $(1-C_{v_\theta})^{-1}$ originally obtained in \cite{DZ-1}. Our proof makes use of the $\overline{\partial}$-nonlinear steepest descent method developed in \cite{DMM}. This method leads to a uniform $L^\infty$ bound of $m_\pm$ (also $\mu$) which are solutions to Problem \ref{prob:NLS.RH1}. Our proof can be easily modified to accommodate the study of perturbations of other completely integrable PDEs. 
    \item[2.]  In Section \ref{sec:priori}, we develop the key \textit{a priori} estimates of the term $F$ given by \eqref{F}. Our proof follows closely \cite[Section 6]{DZ-2} and we make necessary modifications and also implement improvements developed in this paper. 
    \item[3.] In the appendix we obtain new bounds which are originally developed in \cite[Section 5]{DZ-2}. These new bounds rule out the $t$ growth thus reducing the power $l$ of the perturbative term of \eqref{eq: pnls}. The reader can compare our results with \cite[Lemma 5.29, Lemma 5.36]{DZ-2}. 
\end{itemize}

\begin{remark}
We end the introduction with some remarks on notations:
\begin{itemize}
    \item Throughout the text constants $c> 0$ are used generically. $c$ always denotes a constant independent of $x, t, \eta$ and $\rho$.
    \item The $\diamond$ symbol refers to the "dummy" variable of the integrand.
    \item We refer the reader to \cite[(4.1)-(4.3)]{DZ-2} for the definition of the square matrix symbol
$$\Twomat{k}{l}{i}{j}:=\frac{c\eta^k(1+\eta)^l}{(1+t)^i(1-\rho)^j}.$$
We keep this notation for the reader's reference. 
\item  $\Delta$ refers to the difference operator acting on some term $H$ scalar or matrix-valued that contains the reflection coefficient $r_i$, $i=1,2$:
$$\Delta H=H(r_2)-H(r_1).$$
\end{itemize}

\end{remark}

\section{$L^\infty$ estimates through $\overline{\partial}$-nonlinear steepest descent}
\label{sec:dbar}
%In this section, we present our new proof of estimates based on the $\overline{\partial}$-nonlinear steepest descent. 

In \cite{DZ-1}, the following \textit{a priori} estimate for RHP Problem \ref{prob:NLS.RH0} is obtained:
\begin{proposition}[Theorem 1.7 \cite{DZ-1}]
\label{prop:priori}
 Suppose $r \in H_1^{1,0},\|r\|_{H^{1}} \leqslant \lambda,\|r\|_{L^{\infty}} \leqslant \rho<1$. Then for any $x, t \in \mathbb{R}$,  there are
constants $\ell_1 = \ell_1(p), \ell_2 = \ell_2(p) > 0$ and for any $2<p<\infty,\left(1-C_{v_\theta}\right)^{-1}$  exists as bounded operators in $L^p(\mathbb{R})$ and satisfy the bounds

$$
\left\|\left(1-C_{v_\theta}\right)^{-1}\right\|_{L^p \rightarrow L^p} \leqslant K_p,
$$
where
\begin{equation}
K_p=\frac{c(1+\lambda)^{\ell_1}}{(1-\rho)^{\ell_2}}. 
\end{equation}
The constants $K_p$ may be chosen so that $K_p$ is increasing with $p$ and $K_p \geqslant K_2$.
\end{proposition}
In this section, we will make use of the $\overline{\partial}$- steepest descent method to obtain new $L^\infty$ bound on $m_\pm$ and $\mu$ that do not grow in $t$ as $t\to +\infty$. See \cite[lemma 5.17, 5.23]{DZ-2} which play a key role in obtaining the important estimates \cite[(6.5)-(6.6)]{DZ-2}. This in turn leads to a new resolvent operator bound given by the previous proposition. It is known that the computations via the $\overline{\partial}$-nonlinear steepest descent are robust can be adapted to other completely integrable settings. In \cite{CL} and \cite{CLL}, long time asymptotics of the mKdV equation and sine-Gordon equation are studied respectively. $L^\infty$ norm of solutions to the corresponding Riemann-Hilbert problems $m_\pm$ then follow. So we expect that our new proofs here can be adapted to the perturbation theory of other integrable PDEs.  We also point out that the approach we propose here can be adapted to the focusing setting without assuming that $\norm{\rho}{L^\infty} <1$ which only appears the defocussing settings.  We then prove a new type of \textit{a priori} estimate which is an alternative version of Proposition \ref{prop:priori}:
\begin{proposition}
\label{prop:new apriori}
  Suppose $r \in H_1^{1,1},\|r\|_{H^{1,1}} \leqslant \eta,\|r\|_{L^{\infty}} \leqslant \rho<1$. Then for any $x, t \in \mathbb{R}$, and $2<p<\infty,\left(1-C_{v_\theta}\right)^{-1}$  exists as bounded operators in $L^p(\mathbb{R})$ and satisfy the bounds

$$
\left\|\left(1-C_{v_\theta}\right)^{-1}\right\|_{L^p \rightarrow L^p} \leqslant\mathcal{K}_p,
$$
where
\begin{equation}
\label{Kp}
\mathcal{K}_p(\rho, \eta)=(1+\rho)\left(\mathbf{c}_p \left(1+\rho\right)M_{\infty}^2+1\right). 
\end{equation}
The constant $\mathbf{c}_p$ is given by the $L^p$-boundedness of the Cauchy projection.
\end{proposition}
\begin{proof}
Assuming that $m_\pm \in L^\infty$, then by \cite[Proposition 4.5, Proposition 2.6]{DZ-3}, the operator $ \left(I-C_{v_\theta}\right)^{-1} $ exists on $L^p$, and we also recall the following fact from \cite[(7.107)]{Deift}:
    \begin{equation}
        \left(I-C_{v_\theta}\right)^{-1} g=\left(C_{+}\left(g(v_\theta-I) m_{+}^{-1}\right)\right) m_{+} v_{+}^{-1}+g v_{+}^{-1}.
    \end{equation}
    For any $g\in L^p(\bbR)$, from
    \begin{align*}
        \norm{ \left(I-C_{v_\theta}\right)^{-1} g}{L^p}&\leq \norm{\left(C_{+}\left(g(v_\theta-I) m_{+}^{-1}\right)\right) m_{+} v_{+}^{-1}}{L^p}+\norm{g v_{+}^{-1}}{L^p}\\
        &\leq \norm{\left(C_{+}\left(g(v_\theta-I) m_{+}^{-1}\right)\right) }{L^p}\norm{m_{+} v_{+}^{-1}}{L^\infty}+\norm{g v_{+}^{-1}}{L^p}\\
        &\leq \mathbf{c}_p \norm{g}{L^p}\norm{(v_\theta-I) m_{+}^{-1}}{L^\infty} \norm{m_{+} v_{+}^{-1}}{L^\infty}+\norm{g v_{+}^{-1}}{L^p},
    \end{align*}
    we deduce that
    \begin{align}
        \frac{\norm{ \left(I-C_{v_\theta}\right)^{-1} g}{L^p}}{\norm{g}{L^p}}&\leq \mathbf{c}_p \norm{(v_\theta-I) m_{+}^{-1}}{L^\infty} \norm{m_{+} v_{+}^{-1}}{L^\infty}+\norm{v_{+}^{-1}}{L^\infty}\\
        \nonumber
        &\leq \mathbf{c}_p \norm{v_\theta-I}{L^\infty}\norm{v_+}{L^\infty}\norm{m_+}{L^\infty}^2+\norm{v_{+}^{-1}}{L^\infty}\\
        \nonumber
        &\leq (1+\rho)\left(\mathbf{c}_p\norm{m_+}{L^\infty}^2\left(1+\rho\right)+1\right).
    \end{align}
    The conclusion follows if we replace $\norm{m_+}{L^\infty}$ by 
    \begin{equation}
    \label{Minfty}
        M_{\infty}(\rho,\lambda, \eta)=\text{max}\lbrace  C_{\tau m}(\eta), C_m(\rho, \lambda)+1\rbrace
    \end{equation}
   where $C_m(\rho, \lambda)$ and $C_{\tau m}(\eta)$ are given in Proposition \ref{prop:mbound1} and Lemma \ref{lm:mbound2} respectively.%\Red{Maybe we can keep $\norm{m_+}{L^\infty}$ here? and the later on we give the bounds for it.}
\end{proof}
Proposition \ref{prop:priori} above is used in \cite{DZ-2} to deduce various estimates necessary for the proof of \cite[Propsition 2.47]{DZ-2}. In the proof of Proposition \ref{prop:new apriori} we will make use of the $\overline{\partial}$-method as is given by \cite{DMM} to give a simplified proof of Proposition \ref{prop:priori}. Recall that
in \cite{DMM}, a series of transformations are made to the solution $m(z; x, t)$ to Problem \ref{prob:NLS.RH0}
\begin{equation}
\label{m:trans}
    m(z)=E(z)P(\sqrt{8t}(z-z_0))\delta^{\sigma_3}
\end{equation}
where $E$ satisfies the following $\dbar$-integral equation:
\begin{equation}
\label{eq:m3}
    E(z)=I+\frac{1}{\pi} \int_{\mathbb{C}} \frac{\bar{\partial} E(s)}{s-z} d A(s)=I+\frac{1}{\pi} \int_{\mathbb{C}} \frac{E(s) W(s)}{s-z} d A(s).
\end{equation}
Equation \eqref{eq:m3} can be rewritten as 
\begin{equation}
    (I-S)\left[E(z)\right]=I
\end{equation}
where
\begin{equation}
\label{op:S}
    S[f]=\frac{1}{\pi} \int_{\mathbb{C}} \frac{f(s) W(s)}{s-z} d A(s)
\end{equation}
and $W$ is given on \cite[P.10]{DMM18}. Notice that for $t>0$
\begin{align}
   \norm{ S(I)}{L^{\infty} \rightarrow L^{\infty}}  & \leq \left\|\delta^{-2}\right\|_{L^{\infty}\left(\Omega_1\right)} \iint_{\Omega_1} \frac{\left|\bar{\partial} R_1 e^{2 i t \theta}\right|}{|s-z|} d A(s) \\
    \nonumber
    &\leq \left\|\delta^{-2}\right\|_{L^{\infty}\left(\Omega_1\right)}\norm{P}{L^\infty}^2
    \left( \iint_{\Omega_1} \frac{\left|r^{\prime}\right| e^{-t w w}}{|s-z|} d A(s) + \iint_{\Omega_1} \frac{\left|s-z_0\right|^{-1 / 2} e^{-t u v}}{|s-z|} d A(s) \right)\\
    \nonumber
    &\leq t^{-1/4}\left\|\delta^{-2}\right\|_{L^{\infty}\left(\Omega_1\right)}\norm{P}{L^\infty}^2
    \left( c\norm{r}{H^1} +c'\right)\\
    \nonumber
    &\leq  t^{-1/4}\frac{c\rho^2}{(1-\rho)^4}(1+\lambda)\\
    \nonumber
    &=t^{-1/4}C_S(\rho, \lambda)
\end{align}
where $\lambda:=\norm{r}{H^1}$ and 
\begin{equation}
    \delta(z)=e^{C_{\mathbb{R}_{-}+z_0} \log \left(1-|r|^2\right)}=e^{\frac{1}{2 \pi i} \int_{-\infty}^{z_0} \frac{\log \left(1-|r(s)|^2\right)}{s-z} d s}.
\end{equation}
So for sufficiently large $t$, one can  invert the operator \eqref{op:S} by the Neumann series.

\begin{proposition}
    Given $r(z)\in H^{1}$ with $\|r\|_{H^{1}} \leqslant \lambda,\|r\|_{L^{\infty}} \leqslant \rho<1$ and $t^{1/4}>\sqrt{2}C_S$ , then
    \begin{equation}
     \left\vert  E(z)-I\right\vert\leq \frac{C_S}{t^{1/4}-C_S} .
    \end{equation}
    As a consequence, 
    \begin{equation}
    \label{bd:m3}
         \left\vert  E(z)\right\vert \leq \frac{t^{1/4}}{t^{1/4}-C_S}.
          \end{equation}
\end{proposition}
We then obtain the following result that plays a key role in deriving the necessary \textit{a priori} estimates given by \cite{DZ-2}.
\begin{proposition}
\label{prop:mbound1}
  Given $r(z)\in H^{1}$ with $\|r\|_{H^1} \leqslant \lambda, \|r\|_{L^{\infty}} \leqslant \rho<1$ and $t>4C_s^4>0$
    \begin{align}
        \norm{m_\pm(z)-I}{L^\infty}&\leq   \frac{c\rho t^{1/4}}{(1-\rho)^2(t^{1/4}-C_S)}=C_m(\rho, \lambda), \\
        \norm{\mu(z)-I}{L^\infty}&\leq  \frac{c \rho t^{1/4}}{(1-\rho)^2(t^{1/4}-C_S)}=C_\mu(\rho, \lambda).
        \end{align}
\end{proposition}
\begin{proof}
    A consequence of \eqref{bd:m3} and the $L^\infty$-bound of $P$ and $\delta^{\sigma_3}$.
\end{proof}
\begin{lemma}
\label{lm:ED}
    Given $r(z)\in H^{1}$ with $\|r\|_{H^{1}} \leqslant \lambda,\|r\|_{L^{\infty}} \leqslant \rho<1$ and $t>16C_S^4>0$
    \begin{align}
        \norm{\Delta E(z)}{L^\infty}&\leq C_{\delta E}(\rho, \lambda)\norm{\Delta r}{H^{1,1}_1}
    \end{align}
\end{lemma}
\begin{proof}
By the second resolvent identity, 
    \begin{equation}
   \Delta E=\Delta\left({I-S}\right)^{-1}I=(I+S_1+S_1^2+\cdots)(S_2-S_1)(I+S_2+S_2^2+\cdots)
\end{equation}
where for $i=1,2$
\begin{equation}
S_i[f]=\frac{1}{\pi} \int_{\mathbb{C}} \frac{f W_i}{s-z} d A(s).
\end{equation}

Notice that

\begin{equation}
W_i=\bar{\partial} R_{1i} \delta_i^{-2}e^{2it\theta}
\end{equation}

\begin{equation}
\bar{\partial} R_{1i}=(r_i-f_{1i})\bar{\partial}\cos(2\arg z)+\frac{1}{2}\cos(2\arg z)r_i'
\end{equation}

where
\begin{equation}
f_{1i}=\hat{r}_{i0}(z-z_0)^{-2i\nu_i}\delta_i^2
\end{equation}

We set a function $\beta$:

\begin{equation}
\beta(z,z_0)=\int_{-\infty}^{z_0}\frac{1}{2\pi i (s-z)}\left\{ \log(1-|r(s)|^2)-\log(1-|r(z_0)|^2)\chi^0(s-z_0+1) \right\}ds
\end{equation}

\begin{equation}
g=i\nu \left((z-z_0)\ln(z-z_0)-(z-z_0+1)\ln(z-z_0+1)\right)+(\beta(z,z_0)-\beta(z_0,z_0))
\end{equation}
%\Red{Maybe we move the definition below to somewhere above since $\delta_i$ are used above.}
%Write : $\delta = \delta_0 \delta_1$ where
%\begin{equation}
%\delta_0 = e^{\beta(z_0,z_0)+iv(z_0)}(z-%z_0)^{iv(z_0)},\delta_1=e^{g(z,z_0)}
%\end{equation}

We define the parameter $\hat{r}_0$ in the following way

\begin{equation}
\hat{r}_0e^{2iv+2\beta(z_0,z_0)}=r(z_0)
\end{equation}

Thus{\small
\begin{equation}
r(z_0)-f_1=r(z_0)-r(z_0)\exp{\left[2iv((z-z_0)\ln(z-z_0)-(z-z_0+1)\ln(z-z_0+1))+2(\beta(z,z_0)-\beta(z_0,z_0))\right]}.
\end{equation}}
We first show that 
\begin{align}
    \|\beta(\cdot,z_0)\|_{L^{\infty}(R_{z_0})} &\leq c\frac{\|r\|_{H^{1}}}{1-\rho},\\
\|\Delta\beta(\cdot,z_0)\|_{L^{\infty}(R_{z_0})} &\leq c\frac{\|r_2-r_1\|_{H^{1}}}{1-\rho},\\
\|\beta(z,z_0)-\beta(z_0,z_0)\|_{L^{\infty}(R_{z_0})} &\leq c\frac{\|r\|_{H^{1}}}{1-\rho}(z-z_0)^{1/2},\\
\|\Delta(\beta(z,z_0)-\beta(z_0,z_0))\|_{L^{\infty}(R_{z_0})} & \leq c\frac{\|r_2-r_1\|_{H^{1}}}{1-\rho}(z-z_0)^{1/2}.
\end{align}
And before doing so, we recall the following useful properties from \cite{DZ-2}:
\begin{equation}
|\log(1-|r(z)|^2)| \leq \frac{|r(z)|^2}{1-|r(z)|^2}.
\end{equation}

\begin{equation}
\| \log(1-|r(z)|^2)\|_{L^2} \leq \frac{\|r\|_{L^\infty}\|r\|_{L^2}}{1-\|r\|_{L^\infty}}.
\end{equation}

\begin{equation}
\left\vert \log \left( \frac{1-|r_2(z)|^2}{1-|r_1(z)|^2}\right)\right\vert \leq \frac{2\rho}{1-\rho}|r_2(z)-r_1(z)|.
\end{equation}

\begin{equation}
   \left \Vert \log \frac{1-\left\vert r_2(z)\right\vert^2}{1-\left\vert r_1(z)\right\vert^2}-\log \frac{1-\left\vert r_2(0)\right\vert^2}{1-\left\vert r_1(0)\right\vert^2} \chi \right \Vert_{L^2} \leq \frac{c\rho}{1-\rho}\| r_2 - r_1\|_{H^{1}}.
\end{equation}

\begin{align}
&\left\Vert \frac{d}{ds}\left\{ \log (1-|r(s)|^{2})-\log (1-|r(z_{0})|^{2})\chi
^{0}(s-z_{0}+1)\right\} \chi _{(-\infty ,z_{0})}\right\Vert _{L_{R_{z_{0}}}^{2}} \\
\nonumber
&\leq \Vert \frac{r^{\prime }\bar{r}+r\bar{r}^{\prime }}{1-\left\vert
r\left( s\right) \right\vert ^{2}}-\log (1-|r(z_{0})|^{2})\chi ^{0}\Vert
_{L_{R_{z_{0}}}^{2}} \\
\nonumber
&\leq \frac{1}{1-\rho }\left\Vert r\right\Vert
_{L_{R_{z_{0}}}^{2}}\left\Vert r^{\prime }\right\Vert _{L_{R_{z_{0}}}^{2}}+%
\frac{c}{1-\rho }|r(z_{0})| \leq \frac{c}{1-\rho} \|r\|_{H^{1}}
\end{align}

\begin{align}
&\left\Vert \frac{d}{ds}\left\{ \frac{\log (1-|r_{2}(s)|^{2})}{\log
(1-|r_{1}(s)|^{2})}-\frac{\log (1-|r_{2}(z_{0})|^{2})}{\log
(1-|r_{1}(z_{0})|^{2})}\chi ^{0}(s-z_{0}+1)\right\} \chi _{(-\infty
,z_{0})}\right\Vert _{L_{R_{z_{0}}}^{2}} \\
\nonumber
&\leq \frac{1}{1-\rho }\left\Vert r_{2}^{\prime }\bar{r}_{2}+r_{2}\bar{r}%
_{2}^{\prime }-r_{1}^{\prime }\bar{r}_{1}-r_{1}\bar{r}_{1}^{\prime
}\right\Vert _{L_{R_{z_{0}}}^{2}}+\frac{c}{1-\rho }\left\vert r_{2}\left( z_{0}\right)
-r_{1}\left( z_{0}\right) \right\vert \\
\nonumber
&\leq \frac{c}{1-\rho }\left\Vert r_{2}-r_{1}\right\Vert _{H^{1}}.
\end{align}
\begin{equation}
\label{est:Ddelta}
|\Delta \delta^{-2}| \leq c\left ( \frac{\lambda}{(1-\rho^2)}+\frac{\rho}{1-\rho}\left | 1+(1+z)\ln \left | \frac{z}{1+z}\ \right |  \right |  \right )\left \| r_2-r_1 \right \|_{H^{1}} 
\end{equation}
According to the Lemma 23.3 in \cite{BDT}, we can deduce that

\begin{equation}
|\beta(z,z_0)| \leq \left\Vert \left\{ \log(1-|r(s)|^2)-\log(1-|r(z_0)|^2)\chi^0(s-z_0+1) \right\} \chi_{(-\infty,z_0)}\right\Vert_{H^{1}}
\end{equation} 

\begin{equation}
|\beta(z,z_0)-\beta(z_0,z_0)| \leq \sqrt{2} \left\Vert \frac{d}{ds}\left\{ \log(1-|r(s)|^2)-\log(1-|r(z_0)|^2)\chi^0(s-z_0+1) \right\} \chi_{(-\infty,z_0)}\right\Vert_{L_{R_{z_0}}^2}|z-z_0|^{1/2}
\end{equation}

Thus we deduce that

\begin{equation}
\|\beta(\cdot,z_0)\|_{L^{\infty}(R_{z_0})} \leq c\frac{\|r\|_{H^{1}}}{1-\rho}.
\end{equation}

For $i=1,2$, notice that

\begin{equation}
r_i(z_0)-f_{1i}=r_i(z_0)-r_i(z_0)\exp{\left[2i\nu((z-z_0)\ln(z-z_0)-(z-z_0+1)\ln(z-z_0+1))+2(\beta_i(z,z_0)-\beta_i(z_0,z_0))\right]}
\end{equation}

and set
$$g_i=i\nu((z-z_0)\ln(z-z_0)-(z-z_0+1)\ln(z-z_0+1))+(\beta_i(z,z_0)-\beta_i(z_0,z_0)).$$

Then we obtain
\begin{eqnarray*}
|\Delta(r(z_0)-f_1)| &\leq& |\Delta r(z_0)(1-e^{2g_i})+r(z_0) e^{2g_i} \Delta (\beta(z,z_0)-\beta(z_0,z_0)))| \\
&\leq& |\Delta r(z_0)||e^{2g_i}-1|+|r(z_0)||e^{2g_i}||\Delta (\beta(z,z_0)-\beta(z_0,z_0))|\\
&\leq& |\Delta r(z_0)||2g|+|r(z_0)||2g||\Delta (\beta(z,z_0)-\beta(z_0,z_0))|+|r(z_0)||\Delta (\beta(z,z_0)-\beta(z_0,z_0))|.
\end{eqnarray*}

We also observe that 

\begin{equation}
|g|\leq c |z-z_0|^{1/2} + \frac{c}{1-\rho} \|r\|_{H^{1}} |z-z_0|^{1/2},
\end{equation}
and 
\begin{equation}
|\Delta (\beta(z,z_0)-\beta(z_0,z_0))| \leq \frac{c}{1-\rho} \|r_1-r_2\|_{H^{1}} |z-z_0|^{1/2}.
\end{equation}
Thus
\begin{equation}
|\Delta(r(z_0)-f_1)|\leq c |\Delta r(z_0)||z-z_0|^{1/2} + c \|r_1-r_2\|_{H^{1}} |z-z_0|+ c \|r_1-r_2\|_{H^{1}} |z-z_0|^{1/2}.
\end{equation}

Then we can get the following estimate:
\begin{eqnarray}
\label{est:dR}
|\bar{\partial} \Delta R_1| &\leq& \frac{c_1}{|z-z_0|}\left( |\Delta(r-r(z_0))|+|\Delta(r(z_0)-f_1)| \right)+c_2|r_1'-r_2'|\\
\nonumber
&\leq& C_1 \| r_2-r_1 \|_{H^{1}}|z-z_0|^{-1/2}+C_2 \| r_2-r_1 \|_{H^{1}}+C_3 |r_2'-r_1'|.
\end{eqnarray}

Since
\begin{equation}
\Delta W=\bar{\partial} \Delta R_1 \delta^{-2}e^{2it\theta}+ \bar{\partial} R_1 \Delta \delta^{-2}e^{2it\theta}
\end{equation}

we then calculate $\Delta S(f)$:

\begin{equation}
|\Delta S(f)| \leq c(P_1+P_2)
\end{equation}
and will show that
\begin{eqnarray}
\label{P_1}
P_1 &=& \left \vert \int_{\mathbb{C}} \frac{f \bar{\partial} \Delta R_1 \delta^{-2}e^{2it\theta}}{s-z} d A(s) \right \vert \\
\nonumber
&\leq& C\|f\|_{L^\infty}\|\delta^{-2}\|_{L^{\infty}}\frac{\|\Delta r\|_{H^{1,1}}}{t^{1/4}},
\end{eqnarray}

\begin{eqnarray}
\label{P2}
P_2 &=& \left \vert \int_{\mathbb{C}} \frac{f \bar{\partial} R_1 \Delta \delta^{-2}e^{2it\theta}}{s-z} d A(s) \right \vert\\
\nonumber
& \leq & C\|f\|_{L^\infty}\norm{r'}{L^2}\left \| r_2-r_1 \right \|_{H^{1}}\left(\frac{\log t}{t^{1/4}}+ t^{-1/2-1/2q}\right).
\end{eqnarray}
The estimate of \eqref{P_1} is equivalent to that of \eqref{op:S}. We will only focus on \eqref{P2} and integrate on $\Omega_1$  in Figure \ref{fig 1}.
{
\SixMatrix{Six regions}{fig R-2+}
	{\twomat{1}{0}{R_1 e^{2i\theta}\delta^{-2}}{1}}
	{\twomat{1}{-R_3 e^{-2i\theta}\delta^2}{0}{1}}
	{\twomat{1}{0}{R_4 e^{2i\theta}\delta^{-2}}{1}}
	{\twomat{1}{-R_6 e^{-2i\theta}\delta^2}{0}{1}}
}

Recall \eqref{est:Ddelta}. For brevity we set $z_0=0$ and only show the following calculation. For $|z|<1$,
\begin{align}
    \left\|\frac{\log |s| }{\sqrt{|s|}}\right\|_{L^p(v, +\infty)}&\leq \left(\int_{v}^{+\infty}\frac{|\log v|^p}{\left(u^2+v^2\right)^{p / 4}} d u\right)^{1 / p} \\
    \nonumber
    &\leq v^{1 / p-1 / 2}\left(\int_1^{\infty} \frac{|\log v|^p}{\left(1+w^2\right)^{p / 4}} d w\right)^{1 / p}\\
    \nonumber
    &\leq C v^{1 / p-1 / 2}|\log v|.
\end{align}
Also recall that 
\begin{align*}
    \left\|\frac{1}{|s-z|}\right\|_{L^q(v, +\infty)}&=\left(\int_{v}^{+\infty} \frac{1}{\left((u-\alpha)^2+(v-\beta)^2\right)^{q / 2}} d u\right)^{1 / q}\\
    &\leq\left(\int_{\mathbb{R}} \frac{1}{\left(s^2+(v-\beta)^2\right)^{q / 2}} d s\right)^{1 / 2}\\
    &\leq  c|v-\beta|^{1 / q-1} .
\end{align*}
We now combine the above estimates to see that
\begin{align*}
   & \int_0^{\infty} e^{-t v^2}\left\|\frac{\log s}{\sqrt{|s|}}\right\|_{L^p(v, +\infty)}\left\|\frac{1}{|s-z|}\right\|_{L^q(v, +\infty)} d v\\
    &\quad \leq  \int_0^{\infty} e^{-t v^2} v^{1 / p-1 / 2}|\log v||v-\beta|^{1 / q-1} d v\\
&= \int_0^\beta e^{-t v^2}|\log( v)| v^{1 / p-1 / 2}(\beta-v)^{1 / q-1} d v\\
    &\quad +\int_\beta^{\infty} e^{-t v^2}|\log( v)| v^{1 / p-1 / 2}(v-\beta)^{1 / q-1} d v\\
    &=\mathcal{I}_1+\mathcal{I}_2.
\end{align*}
For $\beta<1$, 
\begin{align*}
    \mathcal{I}_1&\leq \int_0^1 \sqrt{\beta} e^{-t \beta^2 w^2} |\log(w)| w^{1 / p-1 / 2}(1-w)^{1 / q-1} d w+\int_0^1 \sqrt{\beta} e^{-t \beta^2 w^2} |\log\beta| w^{1 / p-1 / 2}(1-w)^{1 / q-1} d w\\
    & \leq  c t^{-1 / 4} \left( \int_0^1 |\log(w)| w^{1 / p-1}(1-w)^{1 / q-1} d w \right)+ c t^{-1 / 8}\left(\beta^{1/4}\log\beta \int_0^1 w^{1 / p-3/4}(1-w)^{1 / q-1} d w \right)\\
    &\leq C t^{-1/8}.
\end{align*}
For $\beta>1$, we make use of the fact that $\log\beta<\beta^{1/8}$ to deduce that 
\begin{align*}
     \mathcal{I}_1& \leq ct^{-1/4}+\int_0^1 {\beta}^{5/8} e^{-t \beta^2 w^2} w^{1 / p-1 / 2}(1-w)^{1 / q-1} d w\\
     &\leq ct^{-1/4}+ct^{-5/16}\int_0^1   w^{1 / p-9/ 8}(1-w)^{1 / q-1} d w\\
     &\leq Ct^{-1/4}.
\end{align*}
Moreover,
\begin{align*}
    \mathcal{I}_2&=\int_0^{+\infty} e^{-t(w+\beta)^2}|\log(w+\beta )|(w+\beta)^{1 / p-1 / 2} w^{1 / q-1} d w\\
     &=\int_0^{1-\beta} e^{-t(w+\beta)^2}|\log(w+\beta )|(w+\beta)^{1 / p-1 / 2} w^{1 / q-1} d w+ \int_{1-\beta}^{+\infty} e^{-t(w+\beta)^2}|\log(w+\beta )|(w+\beta)^{1 / p-1 / 2} w^{1 / q-1} d w\\
   &\leq   \int_0^{1-\beta} e^{-t w^2}|\log(w )| w^{-1/2} d w+\int_{1-\beta}^{+\infty} e^{-t(w+\beta)^2}(w+\beta)^{1 / p+1 / 2}w^{1 / q-1} dw\\
   &\leq \int_0^{+\infty} e^{-t w^2}|\log(w )| w^{-1/2} d w+\norm{ e^{-t(w+\beta)^2/2}(w+\beta)}{L^\infty}\int_{0}^{+\infty}e^{-t w^2/2}w^{1/q-1}dw\\
   &\leq c \log t/t^{1/4}+ c t^{-1/2-1/2q}.
\end{align*}
%\Red{Do not feel this is complete. Let's talk about this.}

%\begin{eqnarray*}
 %  \left\vert\Delta \delta^{\sigma_3}\right \vert &\leq& \max_{0\leq T \leq 1}\left\vert e^{C_{\mathbf{R}_-}(\log(1-\left\vert r_2 \right\vert^2))(1-T)+\log(1-\left\vert r_1\right\vert^2)T}\right\vert\left\vert C_{\mathbf{R}_-} \log \frac{1-\left\vert r_2\right\vert^2}{1-\left\vert r_1\right\vert^2}\right\vert \\
 %  &\leq& \left\vert C_{\mathbf{R}_-} \log \frac{1-\left\vert r_2\right\vert^2}{1-\left\vert r_1\right\vert^2}\right\vert
%\end{eqnarray*}

%\begin{equation}
  % \left\vert \log \frac{1-\left\vert r_2\right\vert^2}{1-\left\vert r_1\right\vert^2}\right\vert \leq \frac{2\rho}{1-\rho}\left\vert r_2 - r_1\right\vert
%\end{equation}
\end{proof}
\begin{lemma}
\label{lm:mbound2}
   Given $r(z)\in H^{1,1}_1$ with $\|r\|_{H^{1,1}} \leqslant \eta$ and $\sqrt{2}C_S^4>t>0$
    \begin{align}
        \norm{m_\pm(z)}{L^\infty}&\leq C_{\tau m}(\eta), \\
        \norm{\mu(z)}{L^\infty}&\leq  C_{\tau \mu}(\eta).
    \end{align}
    \end{lemma}
\begin{proof}
Suppose that $t<C_s^{4}$. In \cite{Zhou98}, it is proven that 
\begin{align}
    \norm{q(x,t)}{L^{2,1}}\leq t \norm{r}{H^{1,1}}.
\end{align}
Thus by the standard Volterra theory, we can obtain that 
\begin{equation}
    \norm {m(z;x,t)}{L^\infty} \leq \exp\left(4C_S^4 \norm{r}{H^{1,1}}\right):=C_{\tau m}(\eta).
\end{equation}
\end{proof}
\begin{lemma}
   Given $r(z)\in H^{1,1}_1$ with $\|r\|_{H^{1,1}} \leqslant \eta$, $\|r\|_{L^{\infty}} \leqslant \rho<1$ and $t>0$
    \begin{align}
        \left\vert \Delta m_\pm(z)\right\vert&\leq C_{\Delta m}(\rho,\eta) \left | 1+(1+z)\ln \left | \frac{z}{1+z}\ \right |  \right |\norm{\Delta r}{H^{1,1}_1},\\
        \left\vert\Delta \mu(z)\right\vert&\leq C_{\Delta \mu} (\rho,\eta) \left | 1+(1+z)\ln \left | \frac{z}{1+z}\ \right |  \right |\norm{\Delta r}{H^{1,1}_1}.
    \end{align}  
    
\end{lemma}
\begin{proof}
These are estimates are     a combination of \eqref{est:Ddelta}, the $L^\infty$ bound of the parabolic cylinder function $P$ and Lemma \ref{lm:mbound2}.
\end{proof}

\section{A priori estimates}
\label{sec:priori}

%\Red{We need to be careful here. in D-Z, this symbol is defined with $(1-\rho)^{-j}$. Do we use a different factor here?}

The goal of this section is obtaining the following key estimate that is sufficient for proving our Theorem \ref{main}. Throughout the section we assume that $\norm{r_i}{H^{1,1}}=\eta$ and $\norm{r_i}{L^\infty}=\rho<1$ for $i=1,2$. Compared with \cite[Section 6]{DZ-2}, we make use of newly obtained resolvent operator estimate in Proposition \ref{prop:new apriori} and also new estimates on the $\widetilde{L}$ operator given by \eqref{est LG2}-\eqref{est:DLG1}.
\begin{proposition}
\label{prop:G-estimate}
Let $G$ to be the perturbation term \eqref{G}. Set
$$
    F=P_{12} \int e^{-i \theta} m_{-}^{-1} G m_{-} dy
$$
where $P_{12}$ denotes the projection onto the $(1,2)$ entry of the $2\times 2$ matrix-valued  equation \eqref{eq:rF}.
Then $F$ has the following estimates:
\begin{equation}
\label{est:F}
    \|F\|_{H^{1,1}}\leq \frac{C^{1,1}_F(\rho, \eta)}{(1+t)^{l/2-1/2}}
\end{equation}
\begin{equation}
\label{est:DF}
    \|\Delta F\|_{H^{1,1}}\leq \frac{C^{1,1}_{\Delta F}(\rho, \eta)}{(1+t)^{l/2+1/2p-3/4}}\|\Delta r\|_{H^{1,1}}.\
\end{equation}
where the two constants $C^{1,1}_F(\rho, \eta)$ and $C^{1,1}_F(\rho, \eta)$ are uniform in $x$ and $t$ and are monotonic in $\rho$ and $\eta$.
\end{proposition}
\begin{proof}
Because of the fact that $\|f\|_{H^{1,1}} = \|f\|_{L^2} + \|x f\|_{L^2} +\|\partial_x f\|_{L^2}$, we estimate each norm of the term $F$ and $\Delta F$ separately. \eqref{est:F} follows from Lemma \ref{le:FL2}, Lemma \ref{le:zF2} and Lemma \ref{le:dF}. \eqref{est:DF} follows from Lemma \ref{le:DF}, Lemma \ref{le:DzF} and Lemma \ref{le:dDF}.
\end{proof}
\begin{lemma}The $L^2$ norm of $F$ can be estimated as
\label{le:FL2}
\begin{equation}
\|F\|_{L^2} \leq \Twomat{l+3}{l+1}{(l+1)/2}{5l+5}\mathcal{K}_4^2:=\frac{C_{F}^{0,0}(\rho, \eta)}{(1+t)^{(l+1)/2}}
    \end{equation}
    where $\mathcal{K}_4$ is given by \eqref{Kp} and the constant $C^{0,0}_F(\rho, \eta)$ is uniform in $x$ and $t$ and are monotonic in $\rho$ and $\eta$.
\end{lemma}

\begin{proof}
Similar to \cite[(6.9)]{DZ-2}, we decompose $F$ into the following parts:
\begin{align}
\label{eq:decomF}
    F&=P_{12}\int e^{i \theta}G dy + P_{12}\int e^{i \theta}(m_-^{-1}-I)G dy+P_{12}\int e^{i \theta}G (m_--I) dy + P_{12}\int e^{i \theta}(m_-^{-1}-I)G (m_--I) dy\\
    \nonumber
    &=F^{(1)}+F^{(2)}+F^{(3)}+F^{(4)}.
\end{align}

By Minkowski's inequality, and the estimate of $\|q\|_{L^{\infty}} $ given by \cite[(4.19)]{DZ-2}, we deduce that

\begin{eqnarray*}
\|F\|_{L^2} & \leq & c_1 \|a(x) |q|^l q\|_{L^2} + c_2 \|a(x) |q|^l q\|_{L^1}\|m_--I\|_{L^2(dz) \otimes L^ \infty(dx)} \\
&&+ c_3 \|a(x) |q|^l q\|_{L^2}\|m_-^{-1}-I\|_{L^2(dz) \otimes L^ \infty(dx)} + c_4 \|a(x) |q|^l q\|_{L^2}\|m_-^{-1}-I\|_{L^4(dz) \otimes L^ \infty(dx)}^2 \\
&\leq& \Twomat{l+1}{l+1}{(l+1)/2}{5l+5}+\Twomat{l+2}{l+1}{(l+1)/2}{5l+6}+\Twomat{l+2}{l+1}{(l+1)/2}{5l+6}\\
&&+\Twomat{l+3}{l+1}{(l+1)/2}{5l+6}\mathcal{K}_4^2\\
&\leq& \Twomat{l+1}{l+1}{(l+1)/2}{5l+6}\mathcal{K}_4^2
\end{eqnarray*}as claimed.
\end{proof}

\begin{lemma}The weighed $L^2$ norm of $F$ can be estimated as
\label{le:zF2}
\begin{equation}
    \|z F\|_{L^2} \leq \Twomat{l}{l+3}{l/2}{5l+11}\mathcal{K}_4^2:=\frac{C^{0,1}_{F}(\rho,\eta)}{(1+t)^{l/2}}
    \end{equation}
    where $\mathcal{K}_4$ is given by \eqref{Kp} and the constant $C^{0,1}_F(\rho, \eta)$ is uniform in $x$ and $t$ and are monotonic in $\rho$ and $\eta$.
\end{lemma}

\begin{proof}
We use the decomposition of $F$ given by \eqref{eq:decomF}.
By using the estimate of $\|q_x\|_{L^2}$ given by \cite[Lemma 5.24]{DZ-2}, we get:
\begin{eqnarray*}
    \norm{z F^{(1)}}{L^2(dz)} =c \|\partial_x (a(x) |q|^l q)\|_{L^2} &\leq& c \|\partial_x a(x)\|_{L^2} \|q\|_{L^\infty}^{l+1} + c\|a(x)\|_{L^\infty} \|q\|_{L^\infty}^l \|\partial_x q\|_{L^2}\\
    &\leq& \Twomat{l+1}{l+1}{(l+1)/2}{5l+5}+\Twomat{l+1}{l}{l/2}{5l+1/2}\\
    &\leq& \Twomat{l+1}{l+1}{l/2}{5l+5}.
\end{eqnarray*}

The estimate of $zF^{(2)}$ is the same as $zF^{(3)}$. Integrating by parts on $zF^{(3)}$, we have $zF^{(3)}=I_1+I_2$ with
\begin{eqnarray*}
    \norm{I_{1}}{L^2(dz)} = \left \Vert \int e^{i \theta}G_y(m_--I)dy \right \Vert_{L^2} &\leq&c\|G_x\|_{L^1} \|m_--I\|_{L^2(dz) \otimes L^\infty(dx)} \\
    & \leq &c_1\|\partial_x a(x)\|_{L^1} \|q\|_{L^\infty}^{l+1}\|m_--I\|_{L^2(dz) \otimes L^\infty(dx)}\\
    &&+c_2\|a(x)\|_{L^2} \|q\|_{L^\infty}^l \|\partial_x q\|_{L^2}\|m_--I\|_{L^2(dz) \otimes L^\infty(dx)}\\
    &\leq&\Twomat{l+2}{l+1}{(l+1)/2}{5l+6}+\Twomat{l+2}{l}{l/2}{5l+3/2}\\
    &\leq&\Twomat{l+2}{l+1}{(l+1)/2}{5l+6}.
\end{eqnarray*}
\begin{eqnarray*}
   \norm{I_{2}}{L^2(dz)} = \left \Vert \int e^{i \theta}GQ(m_--I)dy \right \Vert_{L^2(dz)} &\leq&c\|GQ\|_{L^1} \|m_--I\|_{L^2(dz) \otimes L^\infty(dx)} \\
    & \leq & \|a(x)\|_{L^1} \|q\|_{L^\infty}^{l+2}\|m_--I\|_{L^2(dz) \otimes L^\infty(dx)}\\
  &  \leq & \Twomat{l+3}{l+2}{(l+2)/2}{5l+11}.
\end{eqnarray*}

By using the estimate of $\||\diamond|^{1/2}(m_--I)\|_{L^4(dz) \otimes L^\infty(dx)}$ given by \cite[Lemma 5.6]{DZ-2}

\begin{eqnarray*}
    \norm{z F^{(4)}}{L^2(dz)} &=&\left \Vert z\int e^{-i\theta}(m_-^{-1}-I)G(m_--I)dy \right \Vert_{L^2}\\
    &\leq&c\|G\|_{L^1} \||z|^{1/2}(m_-^{-1}-I)\|_{L^4(dz) \otimes L^\infty(dx)}\||z|^{1/2}(m_--I)\|_{L^4(dz) \otimes L^\infty(dx)} \\
    & \leq & \Twomat{l+3}{l+3}{(l+1)/2}{5l+8}\mathcal{K}_4^2.
\end{eqnarray*}

Thus we have

\[\|z F\|_{L^2} \leq \Twomat{l+1}{l+3}{l/2}{5l+11}\mathcal{K}_4^2\]
\end{proof} as desired.

\begin{lemma}The $L^2$ norm of the derivative of $F$ can be estimated as
\label{le:dF}
    \begin{equation}
        \|\partial_z F\|_{L^2} \leq \frac{C^{1,0}(\rho,\eta)}{(1+t)^{(l-1)/2}} 
    \end{equation}
    where $\mathcal{K}_4$ is given by \eqref{Kp}, and the constant $C^{1,0}_{\Delta F}(\rho, \eta)$ which is uniform in $x$ and $t$, and is monotonic in $\rho$ and $\eta$.
\end{lemma}

\begin{proof}
Recall the operator $\tilde{L}:=i x \operatorname{ad} \sigma-2 t \partial_x$ given by \cite[(5.13)]{DZ-2}. We decompose $\partial_z F$ into:
\begin{align}
\label{eq: parzF}
    \partial_z F &=c_1 \int e^{i \theta \mathrm{ad} \sigma}(\partial_z - \tilde{L})m_-^{-1} G dy+c_2\int e^{i \theta \mathrm{ad} \sigma}G (\partial_z - \tilde{L})m_- dy - c_3\int e^{i \theta \mathrm{ad} \sigma}m_-^{-1} \tilde{L}G m_- dy \\
    \nonumber
    &= P_1+ P_2+P_3.
\end{align}

According to Appendix \ref{sec:suppesti} and Section \ref{sec:priori}, we can get:

\begin{eqnarray*}
\|P_1\|_{L^2(dz)} = \norm{\int e^{i \theta \mathrm{ad} \sigma}(\partial_z - \tilde{L})m_-^{-1} G dy}{L^{2}} & \leq & c\|G\|_{L_1(dx)}\|m^{-1}\|_{L^\infty(dz) \otimes L^\infty(dx)}\|(\partial_z -\tilde{L})m_-\|_{L^2(dz) \otimes L^\infty(dx)}\\&\leq& \Twomat{l+2}{l+1}{(l+1)/2}{5l+6}M_\infty\Twomat{1}{0}{0}{1}.
\end{eqnarray*}

The second term has the same estimate as $P_1$. Recall  $m_{-}=\tilde{\mu} \delta_{-}^{\sigma_3} v_\theta^{\#}$ given in \cite[(4.12)]{DZ-2}.  $P_3$ can be decomposed as:

\begin{eqnarray}
\label{eq: P3decomp}
\int e^{i \theta \mathrm{ad} \sigma}m_-^{-1} \tilde{L}G m_- dy  &=& -c \int e^{i \theta \mathrm{ad} \sigma}\left(v_\theta^{\#-1}\delta_-^{\sigma_3-1}\tilde{L}G\delta_-^{\sigma_3}v_\theta^{\#}\right) dy\\
\nonumber
&&-c \int e^{i \theta \mathrm{ad} \sigma}\left(v_\theta^{\#-1}\delta_-^{\sigma_3-1}(\tilde{\mu}^{-1}-I)\tilde{L}G\delta_-^{\sigma_3}v_\theta^{\#}\right) dy\\
\nonumber
&&-c \int e^{i \theta \mathrm{ad} \sigma}\left(v_\theta^{\#-1}\delta_-^{\sigma_3-1}\tilde{L}G(\tilde{\mu}-I)\delta_-^{\sigma_3}v_\theta^{\#}\right) dy\\
\nonumber
&&-c \int e^{i \theta \mathrm{ad} \sigma}\left(v_\theta^{\#-1}\delta_-^{\sigma_3-1}(\tilde{\mu}^{-1}-I)\tilde{L}G(\tilde{\mu}-I)\delta_-^{\sigma_3}v_\theta^{\#}\right) dy\\
\nonumber
&&=P_{31}+P_{32}+P_{33}+P_{34}.
\end{eqnarray}We will estimate each piece of these  four terms separately. 
\cite[(6.19), (6.23)]{DZ-2} gives

\begin{equation}
    \left\vert \int e^{isz}(\delta_-^{-2}(z,z_0)-1)dz \right\vert=\left\vert \int_{z_0}^\infty e^{isz}(\delta^{-2}(-2|r|^2+|r|^4))dz \right\vert \leq \frac{c\eta^2}{1-\rho}\frac{1}{1+|s|},
\end{equation}

and

\begin{equation}
    \left\vert \int_{z_0}^{\infty} e^{isz}\delta^2r(z)^2dz \right\vert \leq \frac{c\eta^2}{1-\rho}\frac{1}{1+|s|}.
\end{equation}

Using the inequality above, we obtain:

\begin{equation}
    \left\Vert \int h(x) e^{i z x}(\delta_-^{-2}(z,z_0)-1)dx \right\Vert_{L^2(dz)} \leq \frac{c\eta^2}{1-\rho}\|h\|_{L^p},\quad  1<p<2,
\end{equation}

\begin{equation}
\label{eq: Lpdelta2h}
    \left\Vert \int h(x) e^{i z x}\delta^2(z,z_0)r(z)^2 \chi\left\{x<2\diamond t\right\} dx \right\Vert_{L^2(dz)} \leq \frac{c\eta^2}{1-\rho}\|h\|_{L^p}, \quad 1<p<2.
\end{equation}

So using conclusions from Appendix \ref{sec:suppesti}, we arrive at

\begin{eqnarray}
\label{eq: I31}
\|P_{31}\|_{L^2(dz)} &=& \left\Vert \int e^{i \theta \mathrm{ad} \sigma}\left(v_\theta^{\#-1}\delta_-^{\sigma_3-1}\tilde{L}G\delta_-^{\sigma_3}v_\theta^{\#}\right) dy \right\Vert_{L^2}\\
\nonumber
&=& \left\Vert\int (\tilde{L}G)_{12}e^{-i\theta}\delta_-^{-2}dx-\int_{x<2tz}(\tilde{L}G)_{21}r^2e^{i\theta}\delta^2dx\right\Vert_{L^2}\\
\nonumber
&\leq& c\|\tilde{L}G\|_{L^2}+\frac{c\eta^2}{1-\rho}\|\tilde{L}G\|_{L^p}+\frac{c\eta^2}{1-\rho}\|\tilde{L}G\|_{L^p}\\
\nonumber
&\leq& c \frac{C_{LG2}(\rho, \eta)}{(1+t)^{(l-1)/2}} +c\frac{C_{LG2}^{2-2/p} C_{LG1}^{2/p-1}}{(1+t)^{(l-1)/2}}.
\end{eqnarray}

The other three parts could be calculated directly:

\begin{eqnarray*}
\|P_{32}\|_{L^2(dz)}=\left\Vert\int e^{i \theta \mathrm{ad} \sigma}\left(v_\theta^{\#-1}\delta_-^{\sigma_3-1}(\tilde{\mu}^{-1}-I)\tilde{L}G\delta_-^{\sigma_3}v_\theta^{\#}\right) dy \right\Vert_{L^2} &=& \frac{c}{1-\rho} \|\tilde{L}G\|_{L^1(dx)}\|\tilde{\mu}-I\|_{L^2(dz) \otimes L^\infty(dx)}\\
&\leq& \frac{c}{1-\rho} \frac{C_{LG1}(\rho, \eta)}{(1+t)^{(l-1)/2}} \Twomat{1}{0}{1/4}{2}\mathcal{K}_2.
\end{eqnarray*}

\begin{eqnarray*}
\|P_{33}\|_{L^2(dz)}=\left\Vert \int e^{i \theta \mathrm{ad} \sigma}\left(v_\theta^{\#-1}\delta_-^{\sigma_3-1}\tilde{L}G(\tilde{\mu}-I)\delta_-^{\sigma_3}v_\theta^{\#}\right) dy \right\Vert_{L^2} &=& \frac{c}{1-\rho} \|\tilde{L}G\|_{L^1(dx)}\|\tilde{\mu}-I\|_{L^2(dz) \otimes L^\infty(dx)}\\
&\leq&  \frac{c}{1-\rho} \frac{C_{LG1}(\rho, \eta)}{(1+t)^{(l-1)/2}} \Twomat{1}{0}{1/4}{2}\mathcal{K}_2.
\end{eqnarray*}

\begin{eqnarray*}
\|P_{34}\|_{L^2(dz)}=\left\Vert \int e^{i \theta \mathrm{ad} \sigma}\left(v_\theta^{\#-1}\delta_-^{\sigma_3-1}(\tilde{\mu}^{-1}-I)\tilde{L}G(\tilde{\mu}-I)\delta_-^{\sigma_3}v_\theta^{\#}\right) dy \right\Vert_{L^2} &=& \frac{c}{1-\rho} \|\tilde{L}G\|_{L^1(dx)}\|\tilde{\mu}-I\|_{L^4(dz) \otimes L^\infty(dx)}^2\\
&\leq&  \frac{c}{1-\rho} \frac{C_{LG1}(\rho, \eta)}{(1+t)^{(l-1)/2}} \Twomat{2}{0}{1/4}{4}\mathcal{K}_4^2.
\end{eqnarray*}
Putting estimates together, we obtain the estimate for $P_3$, and then by the estimates for $P_1$ and $P_2$ the desired estimate follows.
\end{proof}
Three lemmas above give us the {\it a priori} estimates for the fixed point argument. Next, we show lemmas to estimate the difference of two solutions.

\begin{lemma}The $L^2$ norm of the difference can be bounded as
\label{le:DF}
    \begin{equation}
        \|\Delta F\|_{L^2} \leq \Twomat{l}{l+3}{\frac{l}{2}+\frac{1}{2p}+\frac{1}{4}}{5l+8}\mathcal{K}_4^3\|\Delta r\|_{H^{1,1}} := \frac{C^{0,0}_{\Delta F}(\rho, \eta)}{(1+t)^{\frac{l}{2}+\frac{1}{2p}+\frac{1}{4}}} \|\Delta r\|_{H^{1,1}}
    \end{equation}
    where $\mathcal{K}_4$ and $\mathcal{K}_p$ are given by \eqref{Kp} and $p>2$, and the constant $C^{0,0}_{\Delta F}(\rho, \eta)$ is uniform in $x$ and $t$ but  monotonic in $\rho$ and $\eta$.
\end{lemma}
\begin{proof}
    Decomposing $F$ into four parts as given by \eqref{eq:decomF}, we estimate $I_i:= \Delta F^{(i)}$, $i=1,...4$ term by term:
    \begin{eqnarray*}
        \norm{I_1}{L^2 (dz)} = c\|\Delta G\|_{L^2} &\leq&c\|a(x)\|_{L^2}\|q\|_{L^\infty}^{l}\|\Delta q\|_{L^\infty}\\
        &\leq&\Twomat{l}{l+3}{\frac{l}{2}+\frac{1}{2p}+\frac{1}{4}}{5l+7}\mathcal{K}_p\|\Delta r\|_{H^{1,1}}
    \end{eqnarray*}
    where we used the estimate of $\|\Delta q\|_{L^\infty}$ and $\|\Delta m\|_{L^2}$ in \cite[Lemma 4.16 and Lemma 5.12]{DZ-2}.
    \begin{eqnarray*}
       \norm{I_2}{L^2 (dz)} & = & c_1\|\Delta G\|_{L^1} \|m_--I\|_{L^2(dz)\otimes L^\infty(dx)}+ c_2\|G\|_{L^1} \|\Delta m_-\|_{L^2(dz)\otimes L^\infty(dx)}\\
        &\leq&\|a(x)\|_{L^1}\|q\|_{L^\infty}^{l}\|\Delta q\|_{L^\infty}\|m_--I\|_{L^2(dz)\otimes L^\infty(dx)}+\|a(x)\|_{L^1}\|q\|_{L^\infty}^{l+1}\|\Delta m_-\|_{L^2(dz)\otimes L^\infty(dx)}\\
        &\leq&\Twomat{l+1}{l+3}{\frac{l}{2}+\frac{1}{2p}+\frac{1}{4}}{5l+8}\mathcal{K}_p\|\Delta r\|_{H^{1,1}}+\Twomat{l+1}{l+2}{(l+1)/2}{5l+7}\|\Delta r\|_{H^{1,1}}\\
        &\leq&\Twomat{l+1}{l+3}{\frac{l}{2}+\frac{1}{2p}+\frac{1}{4}}{5l+8}\mathcal{K}_p\|\Delta r\|_{H^{1,1}}.
    \end{eqnarray*}

    The estimate of
    \[ \norm{I_3}{L^2 (dz)} = c_1\|\Delta G\|_{L^1} \|m_-^{-1}-I\|_{L^2(dz)\otimes L^\infty(dx)}+ c_2\|G\|_{L^1} \|\Delta m_-^{-1}\|_{L^2(dz)\otimes L^\infty(dx)}\]
    is similar to $I_2$.
    
    \begin{eqnarray*}
        \norm{I_4}{L^2 (dz)} & = & c_1\|\Delta m_-^{-1}\|_{L^4(dz)\otimes L^\infty(dx)}\|G\|_{L^1} \|m_--I\|_{L^4(dz)\otimes L^\infty(dx)}\\
        &&+ c_2\|m_-^{-1}-I\|_{L^4(dz)\otimes L^\infty(dx)}\|\Delta G\|_{L^1} \|m_--I\|_{L^4(dz)\otimes L^\infty(dx)}\\
        &&+ c_3\|m_-^{-1}-I\|_{L^4(dz)\otimes L^\infty(dx)}\|G\|_{L^1} \|m_-^{-1}\|_{L^4(dz)\otimes L^\infty(dx)}\\
        &\leq&\Twomat{l+2}{l+2}{(l+1)/2}{5l+5}\mathcal{K}_4^3\|\Delta r\|_{H^{1,1}}+\Twomat{l+2}{l+3}{\frac{l}{2}+\frac{1}{2p}+\frac{1}{4}}{5l+7}\mathcal{K}_4^3\|\Delta r\|_{H^{1,1}}\\
        &&+\Twomat{l+2}{l+2}{(l+1)/2}{5l+5}\mathcal{K}_4^3\|\Delta r\|_{H^{1,1}}\\
        &\leq&\Twomat{l+2}{l+3}{\frac{l}{2}+\frac{1}{2p}+\frac{1}{4}}{5l+5}\mathcal{K}_4^3\|\Delta r\|_{H^{1,1}}.
    \end{eqnarray*}

    And finally we have:

    \[\|\Delta F\|_{L^2} \leq I_1+I_2+I_3+I_4 \leq \Twomat{l}{l+3}{\frac{l}{2}+\frac{1}{2p}+\frac{1}{4}}{5l+8}\mathcal{K}_4^3\|\Delta r\|_{H^{1,1}}\]as claimed.
\end{proof}

\begin{lemma}The weighted $L^2$ norm of the difference can be estimated
\label{le:DzF}
    \begin{equation}
        \|\Delta (z F)\|_{L^2} \leq \Twomat{l}{l+5}{\frac{l}{2}+\frac{1}{2p}-\frac{1}{4}}{5l+13}\mathcal{K}_p\mathcal{K}_4^3\|\Delta r\|_{H^{1,1}} \leq \frac{C^{0,1}_{\Delta  F}(\rho, \eta)}{(1+t)^{\frac{l}{2}+\frac{1}{2p}-\frac{1}{4}}}
    \end{equation}
    where $\mathcal{K}_4$ is given by \eqref{Kp} and $p>2$ and the constant $C^{0,1}_{\Delta F}(\rho, \eta)$ is uniform in $x$ and $t$ and are monotonic in $\rho$ and $\eta$.
\end{lemma}

\begin{proof}
   As is given by \eqref{eq:decomF}, the first part of  $\|\Delta  (z F^{(1)})\|_{L^2}$ is equal to $\|\Delta \partial_xG\|_{L^2}$ according to the Plancherel's theorem. Here we need the estimate of  $\Delta \partial_x q$  which is divided into two parts $\textrm{I}$ and $\textrm{II}$ as is given in \cite[Lemma 5.37]{DZ-2}.

    \begin{eqnarray*}
        \|\Delta(z F^{(1)})\|_{L^2} &=& c\|\Delta \partial_xG\|_{L^2}\\
        \nonumber
        &\leq& c_1\|\partial_x a(x)\|_{L^2} \|q\|_{L^\infty}^l\|\Delta q\|_{L^\infty}+c_2\|a(x)\|_{L^\infty} \|q\|_{L^\infty}^{l-1}\|\partial_x q\|_{L^2}\|\Delta q\|_{L^\infty}\\
        \nonumber
        &&+c_3\|a(x)\|_{L^2} \|q\|_{L^\infty}^l\|\textrm{I}\|_{L^\infty}+c_4\|a(x)\|_{\infty} \|q\|_{L^\infty}^l\|\textrm{II}\|_{L^2}\\
        \nonumber
        &\leq& \Twomat{l}{l+3}{\frac{l}{2}+\frac{1}{2p}+\frac{1}{4}}{5l+7}\mathcal{K}_p\|\Delta r\|_{H^{1,1}}+\Twomat{l}{l+2}{\frac{l}{2}+\frac{1}{2p}-\frac{1}{4}}{5l+\frac{15}{2}}\mathcal{K}_p\|\Delta r\|_{H^{1,1}}\\
        \nonumber
        &&+\Twomat{l+1}{l+4}{\frac{l}{2}+\frac{1}{2p}}{5l+8}\mathcal{K}_p\|\Delta r\|_{H^{1,1}}+\Twomat{l}{l}{\frac{l}{2}}{5l+2}\|\Delta r\|_{H^{1,1}}\\
        \nonumber
        &\leq& \Twomat{l}{l+4}{\frac{l}{2}+\frac{1}{2p}-\frac{1}{4}}{5l+8}\mathcal{K}_p\|\Delta r\|_{H^{1,1}}.
    \end{eqnarray*}
 We then analyze $\|\Delta (z F^{(3)})\|$ since the estimate of $\|\Delta (z F^{(2)})\|$ is similar. Decomposing it as we have done in the proof of Lemma \ref{le:zF2}:

    \begin{eqnarray*}
         \norm{I_1}{L^2 (dz)}=\left\Vert \int e^{-i\theta}(\Delta \partial_y G)(m_--I) dy\right\Vert_{L^2} &\leq & \|\Delta \partial_xG\|_{L^1(dx)}\|m_--I\|_{L^2(dz)\otimes L^\infty(dx)}\\
        &\leq& \Twomat{l+1}{l+4}{\frac{l}{2}+\frac{1}{2p}-\frac{1}{4}}{5l+9}\mathcal{K}_p\|\Delta r\|_{H^{1,1}}
    \end{eqnarray*}
    where  the estimate of $\|\Delta \partial_xG\|_{L^1(dx)}$ is the same as that of $\|\Delta \partial_xG\|_{L^2(dx)}$. Then
    \begin{eqnarray*}
        \norm{I_2}{L^2 (dz)}=\left\Vert \int e^{-i\theta}(\partial_y G)\Delta m_- dy\right\Vert_{L^2} &\leq & \|\partial_xG\|_{L^1(dx)}\|\Delta m_-\|_{L^2(dz)\otimes L^\infty(dx)}\\
        &\leq& \Twomat{l+1}{l+2}{(l+1)/2}{5l+5}\mathcal{K}_2^2\|\Delta r\|_{H^{1,1}}+\Twomat{l+1}{l+1}{l/2}{5l+1/2}\mathcal{K}_2^2\|\Delta r\|_{H^{1,1}}\\
        &\leq& \Twomat{l+1}{l+2}{l/2}{5l+5}\mathcal{K}_2^2\|\Delta r\|_{H^{1,1}}
    \end{eqnarray*}
    and 
    \begin{eqnarray*}
         \norm{I_3}{L^2 (dz)}=\left\Vert \int e^{-i\theta}(\Delta (GQ))(m_--I)dy \right\Vert_{L^2} &\leq & \|\Delta (GQ)\|_{L^1(dx)}\|m_--I\|_{L^2(dz)\otimes L^\infty(dx)}\\
        &\leq& \Twomat{l+2}{l+4}{\frac{l}{2}+\frac{1}{2p}+\frac{3}{4}}{5l+13}\mathcal{K}_p\|\Delta r\|_{H^{1,1}}
    \end{eqnarray*}
    and 
    \begin{eqnarray*}
         \norm{I_4}{L^2 (dz)}=\left\Vert \int e^{-i\theta}(GQ)(\Delta m_-)dy \right\Vert_{L^2} &\leq & \|GQ\|_{L^1(dx)}\|\Delta m_-\|_{L^2(dz)\otimes L^\infty(dx)}\\
        &\leq& \Twomat{l+2}{l+3}{\frac{l+2}{2}}{5l+10}\mathcal{K}_2^2\|\Delta r\|_{H^{1,1}}.
    \end{eqnarray*}

    We finally turn to $\Delta (zF^{(4)})$:

    \begin{eqnarray*}
         \norm{I_5}{L^2 (dz)} & =& \left\Vert z \int e^{-i\theta}(m_-^{-1}-I)G(\Delta m_-) dy\right\Vert_{L^2}\\
        &\leq & \|G\|_{L^1(dx)}\||\diamond|^{1/2}(m_-^{-1}-I)\|_{L^4(dz)\otimes L^\infty(dx)}\||\diamond|^{1/2}\Delta m_-\|_{L^4(dz)\otimes L^\infty(dx)}\\
        &\leq& \Twomat{l+2}{l+4}{(l+1)/2}{5l+7}\mathcal{K}_4^3\|\Delta r\|_{H^{1,1}}.
    \end{eqnarray*}

    \begin{eqnarray*}
         \norm{I_7}{L^2 (dz)} & =& \left\Vert z \int e^{-i\theta}(m_-^{-1}-I)\Delta G(m_--I) dy\right\Vert_{L^2}\\
        &\leq & \|\Delta G\|_{L^1(dx)}\||z|^{1/2}(m_-^{-1}-I)\|_{L^4(dz)\otimes L^\infty(dx)}\||z|^{1/2}(m_--I)\|_{L^4(dz)\otimes L^\infty(dx)}\\
        &\leq& \Twomat{l+2}{l+5}{\frac{l}{2}+\frac{1}{2p}+\frac{1}{4}}{5l+9}\mathcal{K}_p\mathcal{K}_4^2\|\Delta r\|_{H^{1,1}}.
    \end{eqnarray*}

\end{proof}

\begin{lemma}The $L^2$ norm of the difference of the derivative of $F$ can be estimated as:
\label{le:dDF}
    \begin{equation}
        \|\Delta \partial_z F\|_{L^2} \leq \frac{C^{1,0}_{\Delta F}(\rho, \eta)}{(1+t)^{\frac{l}{2}+\frac{1}{2p}-\frac{3}{4}}}\|\Delta r\|_{H^{1,1}}.
    \end{equation}
    where $\mathcal{K}_4$ is given by \eqref{Kp}, $p>2$ and the constant $C^{1,0}_{\Delta F}(\rho, \eta)$ is uniform in $x$ and $t$ and are monotonic in $\rho$ and $\eta$.
\end{lemma}

\begin{proof}
    As is given by \eqref{eq: parzF}, we again obtain the estimates term by term. We only need to calculate one of $\Delta P_1$ and $\Delta P_2$. For $\Delta P_2$, choose $n\geq 3$ so that for ${\Delta P_2}= { I_1} +{ I_2}+ { I_3}$,

    \begin{eqnarray*}
        \norm{I_1}{L^2 (dz)} & =& \left\Vert \int e^{-i\theta \mathrm{ad} \sigma}(\Delta m_-^{-1})G(\partial_z - \tilde{L})m_- dy\right\Vert_{L^2}\\
        &\leq & \|G\|_{L^1(dx)}\|\Delta m_-^{-1}\|_{L^{\infty}(dz)\otimes L^{\infty}(dx)} \| (\partial_z - \tilde{L})m_-\|_{L^2(dz)\otimes L^\infty(dx)}\\
        &\leq& \Twomat{l+2}{l+3}{l/2+1/4}{5l+25/3}\mathcal{K}_{2(n-1)}^2M_{m_{\pm},\infty}(\rho,\lambda, \eta)\|\Delta r\|_{H^{1,1}}.
    \end{eqnarray*}

    \begin{eqnarray*}
        \norm{I_2}{L^2 (dz)} & =& \left\Vert \int e^{-i\theta \mathrm{ad} \sigma}m_-^{-1}G\Delta(\partial_z - \tilde{L})m_- dy\right\Vert_{L^2}\\
        &\leq & \|G\|_{L^1(dx)}\|m_-^{-1}\|_{L^{\infty}(dz)\otimes L^{\infty}(dx)} \|\Delta (\partial_z - \tilde{L})m_-\|_{L^2(dz)\otimes L^\infty(dx)}\\
        &\leq& \Twomat{l+1}{l+4}{l/2+1/4}{5l+25/3}\mathcal{K}_{2(n-1)}^2M_{m_{\pm},\infty}(\rho,\lambda, \eta)\|\Delta r\|_{H^{1,1}}.
    \end{eqnarray*}

   Note that here we make use of weaker estimate of $\|m_-\|_{L^\infty}$ and $\|\Delta(\partial_z - \tilde{L})m_-\|_{L^\infty}$ given by  \cite[Lemma 5.12, Lemma 5.23]{DZ-2}. 

    \begin{eqnarray*}
       \norm{I_3}{L^2 (dz)}& =& \left\Vert \int e^{-i\theta \mathrm{ad} \sigma}m_-^{-1}\Delta G (\partial_z - \tilde{L})m_- dy\right\Vert_{L^2}\\
        &\leq & \|\Delta G\|_{L^1(dx)}\|m_-^{-1}\|_{L^{\infty}(dz)\otimes L^{\infty}(dx)} \|(\partial_z - \tilde{L})m_-\|_{L^2(dz)\otimes L^\infty(dx)}\\
        &\leq& \Twomat{l+1}{l+3}{l/2+1/2}{5l+8}C_{\tau m}(\eta)M_{m_{\pm},\infty}(\rho,\lambda, \eta)\|\Delta r\|_{H^{1,1}}.
    \end{eqnarray*}

    $\Delta P_3$ could be estimated by following the same argument as \eqref{eq: P3decomp}, the first term $\Delta P_{31}$ would be divided into five pieces:

    \begin{eqnarray*}
    \Delta \left(\int e^{i \theta \mathrm{ad} \sigma}\left(v_\theta^{\#-1}\delta_-^{\sigma_3-1}\tilde{L}G\delta_-^{\sigma_3}v_\theta^{\#}\right) dy \right)&=& \int \Delta (\tilde{L}G)_{12}e^{-i\theta}\delta_-^{-2}dx+\int (\tilde{L}G)_{12}e^{-i\theta} \Delta \delta_-^{-2}dx\\
    &&-\int_{x<2tz}\Delta (\tilde{L}G)_{21}r^2e^{i\theta}\delta^2dx\\
    &&-\int_{x<2tz}(\tilde{L}G)_{21} (\Delta r^2) e^{i\theta}\delta^2dx\\
    &&-\int_{x<2tz}(\tilde{L}G)_{21}r^2e^{i\theta} \Delta \delta^2dx.
    \end{eqnarray*}

    Following the conclusion in \eqref{eq: I31}, we have:
    \begin{eqnarray*}
    &&\left\Vert\int \Delta (\tilde{L}G)_{12}e^{-i\theta}\delta_-^{-2}dx-\int_{x<2tz} \Delta (\tilde{L}G)_{21}r^2e^{i\theta}\delta^2dx\right\Vert_{L^2(dz)}\\
    &\leq& c\|\Delta \tilde{L}G\|_{L^2}+\frac{c\eta^2}{1-\rho}\|\Delta \tilde{L}G\|_{L^p}+\frac{c\eta^2}{1-\rho}\|\Delta \tilde{L}G\|_{L^p}\\
    \nonumber
    &\leq& c \frac{C_{\Delta LG2}(\rho, \eta)}{(1+t)^{\frac{l}{2}+\frac{1}{2p}-\frac{3}{4}}}\left\Vert \Delta r\right\Vert _{H^{1,1}} +c\frac{C_{\Delta LG2}^{2-2/p} C_{\Delta LG1}^{2/p-1}}{(1+t)^{\frac{l}{2}+\frac{1}{2p}-\frac{3}{4}}}\left\Vert \Delta r\right\Vert _{H^{1,1}}.
    \end{eqnarray*}

    Following \eqref{eq: Lpdelta2h}, we have:

    \begin{eqnarray*}
    \left\Vert\int_{x<2tz}(\tilde{L}G)_{21} (\Delta r^2) e^{i\theta}\delta^2dx\right\Vert_{L^2}
    &\leq& \frac{c\eta}{1-\rho}\|\Delta r\|_{H^{1,1}}\|\tilde{L}G\|_{L^p}\\
    &\leq&c\frac{C_{LG2}^{2-2/p} C_{LG1}^{2/p-1}}{(1+t)^{(l-1)/2}}\|\Delta r\|_{H^{1,1}}.
    \end{eqnarray*}

    \cite[(6.40)]{DZ-2} gives 

    \begin{equation}
    \left\vert \int e^{isz}\Delta \delta_-^{-2}dz \right\vert \leq \frac{c\eta(1+\eta)}{(1-\rho)^2}\frac{\|\Delta r\|_{H^{1,1}}}{(1+|s|)^{1/\alpha}}, 1<\alpha<\infty,
    \end{equation}
    
    \begin{equation}
    \left\Vert \int h(x) e^{i \diamond x}\Delta \delta_-^{-2}dx \right\Vert_{L^2(dz)} \leq \frac{c\eta(1+\eta)}{(1-\rho)^2}\|h\|_{L^p}\|\Delta r\|_{H^{1,1}},\quad  1<p<2.
    \end{equation}

    Using these inequalities, we have:

    \begin{eqnarray*}
    &&\left\Vert\int (\tilde{L}G)_{12}e^{-i\theta}\Delta \delta_-^{-2}dx-\int_{x<2tz} (\tilde{L}G)_{21}r^2e^{i\theta} \Delta \delta^2dx\right\Vert_{L^2}\\
    &\leq& \frac{c\eta(1+\eta)}{(1-\rho)^2}\|\tilde{L}G\|_{L^p}\|\Delta r\|_{H^{1,1}} + \frac{c\eta ^2}{(1-\rho)^2}\|\tilde{L}G\|_{L^p}\|\Delta r\|_{H^{1,1}}\\
    \nonumber
    &\leq& c\frac{C_{ LG2}^{2-2/p} C_{ LG1}^{2/p-1}}{(1+t)^{(l-1)/2}} \left\Vert \Delta r\right\Vert _{H^{1,1}}.
    \end{eqnarray*}

 This completes the estimation of $\Delta P_{31}$.   We only calculate  $\Delta P_{33}$. And we decompose it into two parts:

    \begin{eqnarray*}
    \Delta P_{331} &=& -c \int e^{i \theta \mathrm{ad} \sigma}\left(v_\theta^{\#-1}\delta_-^{\sigma_3-1}(\Delta \tilde{L}G)(\tilde{\mu}-I)\delta_-^{\sigma_3}v_\theta^{\#}\right) dy \\
    &&-c \int e^{i \theta \mathrm{ad} \sigma}\left(v_\theta^{\#-1}\delta_-^{\sigma_3-1}\tilde{L}G\Delta \tilde{\mu}\delta_-^{\sigma_3}v_\theta^{\#}\right) dy,
    \end{eqnarray*}

    \begin{eqnarray*}
    \Delta P_{332} &=& c\int (\tilde{L}G)_{12}(\tilde{\mu}-I)e^{-i\theta} \Delta \delta_-^{-2}dx\\
    &&+c\int_{x<2tz}(\tilde{L}G)_{21}(\tilde{\mu}-I) (\Delta r^2) e^{i\theta}\delta^2dx\\
    &&+c\int_{x<2tz}(\tilde{L}G)_{21}(\tilde{\mu}-I) r^2e^{i\theta} \Delta \delta^2dx
    \end{eqnarray*}

   then the estimates are given by the following:

    \begin{eqnarray*}
    \|\Delta P_{331}\|_{L^{2}(dz)} &=& \frac{c}{1-\rho}\|\Delta \tilde{L}G\|_{L^1}\|\tilde{\mu}-I\|_{L^2(dz) \otimes L^\infty(dx)}+\frac{c}{1-\rho}\|\tilde{L}G\|_{L^1}\|\Delta \tilde{\mu}\|_{L^2(dz) \otimes L^\infty(dx)}\\
    &\leq& \frac{C_{\Delta LG1}}{(1+t)^{\frac{l}{2}+\frac{1}{2p}-\frac{3}{4}}}\Twomat{1}{0}{1/4}{4}\|\Delta r\|_{H^{1,1}}+\frac{C_{LG1}(\rho, \eta)}{(1+t)^{(l-1)/2}}\Twomat{1}{0}{\frac{1}{2p}}{4}\mathcal{K}_{p}^2\|\Delta r\|_{H^{1,1}},
    \end{eqnarray*}

    \begin{eqnarray*}
    \|\Delta P_{332}\|_{L^{2}(dz)} &=& \frac{c}{1-\rho}\|\tilde{L}G\|_{L^1}\|\tilde{\mu}-I\|_{L^2(dz) \otimes L^\infty(dx)}\|\Delta r\|_{H^{1,1}}\\
    &&+c\|\tilde{L}G\|_{L^1}\|(|\Delta \delta_-^2|+|\Delta \delta_-^{-2}|)(\tilde{\mu}-I)\|_{L^2(dz) \otimes L^\infty(dx)}\\
    &\leq& \frac{C_{LG1}(\rho, \eta)}{(1+t)^{(l-1)/2}}\Twomat{1}{0}{1/4}{4}\|\Delta r\|_{H^{1,1}}+\frac{C_{LG1}(\rho, \eta)}{(1+t)^{(l-1)/2}}\Twomat{1}{0}{\frac{1}{2p}}{4}\mathcal{K}_{p}\|\Delta r\|_{H^{1,1}}.
    \end{eqnarray*}

   We finally get the estimates for $\Delta P_{34}$:

    \begin{eqnarray*}
    \|\Delta P_{34}\|_{L^{2}(dz)} &\leq& \frac{c}{1-\rho}\|\Delta \tilde{L}G\|_{L^1}\|\tilde{\mu}-I\|_{L^4(dz) \otimes L^\infty(dx)}^2\\
    &&+\frac{c}{1-\rho}\|\tilde{L}G\|_{L^1}\|\Delta \tilde{\mu}\|_{L^4(dz) \otimes L^\infty(dx)}\|\tilde{\mu}-I\|_{L^4(dz) \otimes L^\infty(dx)}\\
    &&+\frac{c}{1-\rho}\|\tilde{L}G\|_{L^1}\|\tilde{\mu}-I\|_{L^4(dz) \otimes L^\infty(dx)}^2\|\Delta r\|_{H^{1,1}}\\
    &&+c\|\tilde{L}G\|_{L^1}\|\tilde{\mu}-I\|_{L^4(dz) \otimes L^\infty(dx)}\|(|\Delta \delta_-^2|+|\Delta \delta_-^{-2}|)(\tilde{\mu}-I)\|_{L^4(dz) \otimes L^\infty(dx)}\\
    &\leq& \frac{C_{\Delta LG1}}{(1+t)^{\frac{l}{2}+\frac{1}{2p}-\frac{3}{4}}}\Twomat{1}{0}{1/4}{5}\mathcal{K}_4^2\|\Delta r\|_{H^{1,1}}+\frac{C_{LG1}(\rho, \eta)}{(1+t)^{(l-1)/2}}\Twomat{1}{2}{\frac{1}{2p}+\frac{1}{8}}{7}\mathcal{K}_{p}\mathcal{K}_{4}\|\Delta r\|_{H^{1,1}}\\
    &&+\frac{C_{LG1}(\rho, \eta)}{(1+t)^{(l-1)/2}}\Twomat{1}{0}{1/4}{5}\mathcal{K}_4^2\|\Delta r\|_{H^{1,1}}+\frac{C_{LG1}(\rho, \eta)}{(1+t)^{(l-1)/2}}\Twomat{1}{2}{\frac{1}{2p}+\frac{1}{8}}{7}\mathcal{K}_p\mathcal{K}_{4}\|\Delta r\|_{H^{1,1}}.
    \end{eqnarray*}
The claimed estimate follows from putting all estimates above together.

\end{proof}

\appendix
\section{Supplementary estimates}
\label{sec:suppesti}
In this appendix we give all the necessary estimates needed to obtain \eqref{est:F}-\eqref{est:DF}. Throughout the section we assume that $r(z)\in H^{1,1}_1$ with $\|r\|_{H^{1,1}} \leqslant \eta$, $\|r\|_{L^{\infty}} \leqslant \rho<1$. We will make use of the following properties
\begin{itemize}
    \item[1.] The $L^1$ integrability of $a(x)$;
    \item[2.] The uniform $L^\infty$ bounds of $m_\pm$ and $\mu$ obtained from  Section \ref{sec:dbar}.
\end{itemize}
We will closely follow the proofs in \cite[Section 5]{DZ-2} and make modifications accordingly.
We first recall the following operators:
\begin{align}
    & L:=\partial_z - i(x-2xt) \mathrm{ad} \sigma,\\
    &\tilde{L}:=i x \operatorname{ad} \sigma-2 t \partial_x
\end{align}
and these two operators are related by:
\begin{align}
    (L+2tQ) \mu & =\left(\partial_z-\tilde{L}\right) \mu \\
    (L+2tQ) m_{ \pm} & =\left(\partial_z-\tilde{L}\right) m_{ \pm}.
\end{align}
Following the conventions of  \cite[Section 5]{DZ-2}, we first introduce the notation $Lf_\theta=f'_\theta$, where $f_\theta = e^{i\theta \mathrm{ad}}f$. According to  Equation \eqref{eq:mu}, and the commutation relation
\begin{equation}
    \diamond C^\pm -C^\pm \diamond=-\frac{1}{2\pi i}\int
\end{equation}
we then deduce that
\begin{equation}
    (L+2tQ)\mu =C_{w_\theta}(L+2 t Q) \mu + C_{\omega'_\theta} \mu.
\end{equation}

From the jump relation \eqref{NLS.V}, we deduce that
\begin{equation}
    \partial_z m_+=(\partial_z m_-)v_\theta+m_-(i(x-2tz)\mathrm{ad} v_\theta)
\end{equation}

and in terms of the $L$ operator we have
\begin{equation}
\label{eq:L2tq}
    (L+2tQ) m_+=((L+2tQ) m_-)v_\theta+m_-(\partial_z v)_\theta
\end{equation}
which is an inhomogeneous RHP problem. 

First note that from the results of Section \ref{sec:dbar}, we easily deduce:
\begin{equation}
\label{eq:tildeLmu}
 \|(L+2tQ)\mu\|_{L^2} \leq \frac{c}{1-\rho} \| r'\|_{L^2}\|\mu\|_{L^\infty}
    \leq \Twomat{1}{0}{0}{1} M_{\mu,\infty}(\rho,\lambda, \eta).
\end{equation}

\begin{equation}
\label{eq:tildeLm}
 \|(L+2tQ)m_{\pm}\|_{L^2} \leq c \|(L+2tQ)\mu\|_{L^2}\|r\|_{L^\infty}+\| r'\|_{L^2}\|\mu\|_{L^\infty}
    \leq \Twomat{1}{0}{0}{1} M_{m_{\pm},\infty}(\rho,\lambda, \eta).
\end{equation}

For $\Delta (L+2tQ)\mu$, the analysis is divided into the following parts:
\begin{eqnarray*}
    \Delta (L+2tQ)\mu & =& \Delta \left[\left(1-C_{w_\theta}\right)^{-1} C_{w_\theta^{\prime}} \mu\right]\\
    &=& \left(\Delta\left(1-C_{w_\theta}\right)^{-1}\right) C_{w_\theta^{\prime}} \mu+\left(1-C_{w_\theta}\right)^{-1} C_{\Delta w_\theta^{\prime}} \mu+\left(1-C_{w_\theta}\right)^{-1} C_{w_\theta^{\prime}} \Delta \mu \\
    &=&  \Delta^1 (L+2tQ)\mu+\Delta^2 (L+2tQ)\mu+\Delta^3(L+2tQ)\mu
    \end{eqnarray*}

\begin{align}
\norm{\Delta^1 (L+2tQ)\mu}{L^2} &\leq\Twomat{0}{0}{0}{2}\|\Delta r\|_{H^{1,1}}\Twomat{1}{0}{0}{0}\|\mu\|_{L^\infty}\\
\nonumber
 &\leq \Twomat{1}{0}{0}{2}\|\Delta r\|_{H^{1,1}} M_{\infty}(\rho,\lambda, \eta).
\end{align}

\begin{align}
\norm{\Delta^2 (L+2tQ)\mu}{L^2} &\leq\Twomat{0}{0}{0}{1}\|\Delta r\|_{H^{1,1}}\|\mu\|_{L^\infty}\\
\nonumber
 &\leq \Twomat{0}{0}{0}{1}\|\Delta r\|_{H^{1,1}} M_{\infty}(\rho,\lambda, \eta).
\end{align}

\begin{align}
\norm{\Delta^3 (L+2tQ)\mu}{L^{4/3}} &\leq \norm{\left(1-C_{w_\theta}\right)^{-1}}{L^{4/3} \rightarrow L^{4/3}}\norm{C_{w'_\theta} \Delta\mu }{L^{4/3}}\\
\nonumber
&\leq \mathcal{K}_{4/3} C_{\Delta \mu w' }(\rho, \eta ) \|\Delta r\|_{H^{1,1}}.
\end{align}
    Here we make use of the following fact:
    \begin{align}
       \norm{C_{w'_\theta} \Delta\mu }{L^{4/3}} &\leq \mathbf{c}_{4/3}\norm{r'_\theta \Delta \delta }{L^{4/3}} C_{\Delta m}^2(\eta) \\
       \nonumber
      & \leq \mathbf{c}_{4/3}  C_{\Delta m}^2(\eta)  \left(\int_\mathbb{R} \left\vert  r'_\theta \Delta \delta_- \right\vert^{4/3} dz\right)^{3/4}\\
      \nonumber
      &\leq \mathbf{c}_{4/3}  C_{\Delta m}^2(\eta) 
      \left( \left(\int_\mathbb{R} |r'_\theta|^2 dz\right)^{2/3} \left(\int_\mathbb{R} |\Delta \delta_-|^4dz\right)^{1/3}\right)^{3/4}\\
      \nonumber
      &\leq   \mathbf{c}_{4/3}  C_{\Delta m}^2(\eta)   \frac{\lambda \norm{\Delta r}{H^{1,1}_1}}{1-\rho}\\
      \nonumber
      &:= C_{\Delta \mu w' }(\rho, \eta )\norm{\Delta r}{H^{1,1}_1}
      \end{align}

  % &=&   +\Twomat{0}{0}{0}{1}\|\Delta r\|_{H^{1,1}} \|\mu\|_{L^\infty} \\
  %  &&+\Twomat{0}{0}{0}{1}\Twomat{1}{0}{0}{0}\|\Delta \mu\|_{L^\infty}\\
  %  &\leq&\Twomat{1}{0}{0}{1}C_{\tau \mu}(\eta)\|\Delta r\|_{H^{1,1}}+\Twomat{1}{0}{0}{1}C_{\Delta \mu}^2 (\eta) \norm{\Delta r}{H^{1,1}_1}\int \left | 1+(1+z)\ln \left | \frac{z}{1+z}\ \right |  \right |^2dz\\
 %   & \leq & \Twomat{1}{0}{0}{1} \norm{\Delta r}{H^{1,1}_1} \max{C_{\tau \mu}(\eta), C_{\Delta \mu}^2 (\eta) \int \left | 1+(1+z)\ln \left | \frac{z}{1+z}\ \right |  \right |^2dz}

% For $\Delta (L+2tQ)m_{\pm}$, the analysis is divided into the following parts:
%\begin{eqnarray}
%    \Delta (L+2tQ) m_{\pm} &=& \Delta \left[((L+2tQ) \mu) v_\theta^\pm + \mu v'_\theta \right]\\
%    \nonumber
%    &=& \left(\Delta((L+2tQ) \mu)\right) v_\theta^{\pm}+\left((L+2tQ) \mu\right) \Delta v_\theta^{\pm}+\left(\Delta \mu \right) v'_\theta +\mu \Delta v'_\theta\\
%    \nonumber
%    &=& M_1+M_2+M_3+M_4.
%\end{eqnarray}

Now we consider the term of perturbation:
$G\left( q\right) =-ia\left( x\right) \left\vert q\right\vert ^{l}q$
and deduce various necessary estimates. First recall \cite[P.228]{DZ-2}:
\begin{eqnarray*}
    \tilde{L} \mathbf{Q} & = &-\int\left(((L+2 t Q) \mu) w_\theta-(\mu-I) w_\theta^{\prime}\right)+\int w_\theta^{\prime} \\
    & \equiv & L_Q+L_Q^{\prime}
\end{eqnarray*}
It is straightforward to check that
\begin{eqnarray*}
    \|L_Q\|_{L^\infty} & \leq & \|\left( L+2tQ \right)\mu\|_{L^2}\|r\|_{L^2}+\|\mu-I\|_{L^2}\|r'\|_L^2 \\
    & \leq & \Twomat{2}{0}{0}{1}M_{\infty}(\rho,\lambda, \eta)+\Twomat{2}{0}{0}{1} \\
    &\leq& \Twomat{2}{0}{0}{1}M_{\infty}(\rho,\lambda, \eta)\\
    \|L'_Q\|_{L^2} & \leq & c\eta \leq \Twomat{1}{0}{0}{0},
\end{eqnarray*}
and
\begin{align*}
    \|\Delta L_Q\|_{L^\infty} & \leq  \| \Delta^1 \left(L+2tQ \right)\mu\|_{L^2}\|r\|_{L^2} + \| \Delta^2 \left(L+2tQ \right)\mu\|_{L^2}\|r\|_{L^2}+ \| \Delta^3 \left(L+2tQ \right)\mu\|_{L^{4/3}}\|r\|_{L^4}\\
    &\quad + \|\left(L+2tQ \right)\mu\|_{L^2}\|\Delta r\|_{L^2}+\|\mu-I\|_{L^2}\|\Delta r'\|_L^2+\|\Delta \mu\|_{L^2}\|r'\|_L^2 \\
     &  \leq  \Twomat{1}{0}{0}{2}\|\Delta r\|_{H^{1,1}} M_{\infty}(\rho,\lambda, \eta)+ \Twomat{0}{0}{0}{1}\|\Delta r\|_{H^{1,1}} M_{\infty}(\rho,\lambda, \eta) + C_{\Delta \mu w' }(\rho, \eta )\norm{\Delta r}{H^{1,1}_1}\\
     &\quad +  \Twomat{1}{0}{0}{1} M_{\infty}(\rho,\lambda, \eta)\norm{\Delta r}{H^{1,1}_1}\\
     &=C_{\Delta LQ} \norm{\Delta r}{H^{1,1}_1}.
\end{align*}

For the term $\tilde{L}G$ , we have that
\begin{eqnarray*}
\widetilde{L}G &=&-i\left( ix\mathrm{ad}\sigma-2t\partial_{x}\right) a\left(
x\right) \left\vert q\right\vert^{l}\twomat{0}{q}{-\overline{q}}{0}  \\
&=&-i\twomat{0}{\beta}{-\overline{\beta}}{0}
\end{eqnarray*}
where
\begin{equation}
\beta =a\left( x\right) \left\vert q\right\vert ^{l}\left(
ix-2t\partial _{x}\right) q+la\left( x\right) \left\vert q\right\vert ^{l-2}q%
\mathrm{Re}\left(\bar{q}\left( ix-2t\partial_{x}\right) q\right) +\left\vert q\right\vert ^{l}q\left( ix-2t\partial _{x}\right) a\left( x\right).
\end{equation}
We then obtain the estimate of the $\tilde{L}G$:
{\small \begin{align}
\label{est LG2}
\left\Vert \tilde{L}G\right\Vert _{L^{2}\left( dx\right) } &\leq \left\Vert
a\left( x\right) \left\vert q\right\vert ^{l}\left(ix-2t\partial
_{x}\right) q+la\left( x\right) \left\vert q\right\vert ^{l-2}q\mathrm{Re}
\left( \bar{q}\left( ix-2t\partial _{x}\right) q\right) \right\Vert
_{L^{2}\left( dx\right) }+\left\Vert \left\vert q\right\vert ^{l}q\left(
2t\partial _{x}\right) a\left( x\right) \right\Vert _{L^{2}\left(
dx\right) } \\
\nonumber
        &\leq \left\Vert a\left( x\right) \right\Vert _{L^{2}\left( dx\right)
}\left\Vert q\right\Vert _{L^{\infty }\left( dx\right) }^{l}\left\Vert
L_{Q}\right\Vert _{L^{\infty }\left( dx\right) }+\left\Vert a\left( x\right)
\right\Vert _{L^{\infty }\left( dx\right) }\left\Vert q\right\Vert
_{L^{\infty }\left( dx\right) }^{l}\left\Vert L_{Q^{\prime }}\right\Vert
_{L^{2}\left( dx\right) }\\
\nonumber
 & + t\left\Vert \partial _{x}a\left( x\right) \right\Vert
_{L^{2}\left( dx\right) }\left\Vert q\right\Vert _{L^{\infty }\left(
dx\right) }^{l+1} \\
\nonumber
&\leq \norm{ a(x)}{L^2} \Twomat{l}{l}{l/2}{5l}\Twomat{2}{0}{0}{1}M_{\infty}(\rho,\lambda, \eta) + \norm{a(x)}{L^\infty} \Twomat{l}{l}{l/2}{5l}\Twomat{1}{0}{0}{0} \\
\nonumber
&+ \norm{\partial_x a(x)}{L^1} t\Twomat{l+1}{l+1}{(l+1)/2}{5l+5}\\
\nonumber
&\leq  \Twomat{l+1}{l+1}{(l-1)/2}{5l}M_{\infty}(\rho,\lambda, \eta)\left(\norm{ a(x)}{L^2}+ \norm{a(x)}{L^\infty} +\norm{\partial_x a(x)}{L^1}\right)\\
\nonumber
&= \frac{C_{LG2}(\rho, \eta)}{(1+t)^{(l-1)/2}}.
\end{align}}
And
{\small \begin{align}
\label{est:LG1}
\left\Vert \tilde{L}G\right\Vert _{L^{1}\left( dx\right) } &\leq \left\Vert
a\left( x\right) \left\vert q\right\vert ^{l}\left( ix-2t\partial
_{x}\right) q+l a\left( x\right) \left\vert q\right\vert ^{l-2}q\mathrm{Re}
\left( \bar{q}\left( ix-2t\partial _{x}\right) q\right) \right\Vert
_{L^{1}\left( dx\right) }+\left\Vert \left\vert q\right\vert ^{l}q\left(
2t\partial _{x}\right) a\left( x\right) \right\Vert _{L^{1}\left(
dx\right) } \\
\nonumber
&\leq \left\Vert a\left( x\right) \right\Vert _{L^{1}\left( dx\right)
}\left\Vert q\right\Vert _{L^{\infty }\left( dx\right) }^{l}\left\Vert
L_{Q}\right\Vert _{L^{\infty }\left( dx\right) }+\left\Vert a\left( x\right)
\right\Vert _{L^{2}\left( dx\right) }\left\Vert q\right\Vert _{L^{\infty
}\left( dx\right) }^{l}\left\Vert L_{Q^{\prime }}\right\Vert _{L^{2}\left(
dx\right) }\\
\nonumber
&  +t\left\Vert
\partial _{x}a\left( x\right) \right\Vert _{L^{1}\left( dx\right)
}\left\Vert q\right\Vert _{L^{\infty }\left( dx\right) }^{l+1} \\
\nonumber
&\leq \frac{1}{t^{l/2}} \left\Vert  a\left( x\right) \right\Vert _{L^{1}\left( dx\right) }  \Twomat{2}{0}{0}{1}M_{\infty}(\rho,\lambda, \eta)+ \frac{1}{t^{l/2}}\left\Vert a\left( x\right) \right\Vert
_{L^{2}\left( dx\right) } \Twomat{1}{0}{0}{0}+\frac{1}{t^{(l-1)/2}}\left\Vert \partial _{x}a\left( x\right) \right\Vert _{L^{1}\left(
dx\right) }\\
\nonumber
&\leq   \Twomat{2}{0}{{(l-1)/2}}{1}\left(\norm{a(x)}{L^1}+\norm{\partial_x a(x)}{L^1}\right)M_{\infty}(\rho,\lambda, \eta)\\
\nonumber
&=\frac{C_{LG1}(\rho, \eta)}{(1+t)^{(l-1)/2}}.
\end{align}}
Finally, the difference, we first have
\begin{align}
\label{est:DLG2}
\Vert \Delta \tilde{L}G\Vert _{L^{2}(dx)} &\leqslant \left\Vert
a(x)\Delta \left( |q|^{l}\right) \tilde{L}\mathbf{Q}\right\Vert
_{L^{2}(dx)}+\left\Vert a(x)\Delta \left( |q|^{l-2}q\right) \tilde{L}%
\mathbf{Q}\right\Vert _{L^{2}(dx)}+\left\Vert \Delta \left( |q|^{l}q\right) 
\tilde{L}a(x)\right\Vert _{L^{2}(dx)}\\
\nonumber
&\quad +\left\Vert a(x)\left\vert q\right\vert
^{l}\Delta \tilde{L}\mathbf{Q}\right\Vert _{L^{2}(dx)}  \\
\nonumber
&\leqslant \left\Vert a\left( x\right) \right\Vert _{L^{2}\left( dx\right)
}\left\Vert q\right\Vert _{L^{\infty }\left( dx\right) }^{l-1}\left\Vert
\Delta q\right\Vert _{L^{\infty }\left( dx\right) }\left\Vert
L_{Q}\right\Vert _{L^{\infty }\left( dx\right) } +\left\Vert a\left( x\right) \right\Vert _{L^{\infty }\left( dx\right)
}\left\Vert q\right\Vert _{L^{\infty }\left( dx\right) }^{l-1}\left\Vert
\Delta q\right\Vert _{L^{\infty }\left( dx\right) }\left\Vert L_{Q^{\prime
}}\right\Vert _{L^{2}\left( dx\right) } \\
\nonumber
&+\left\Vert a\left( x\right) \right\Vert _{L^{2}\left( dx\right)
}\left\Vert q\right\Vert _{L^{\infty }\left( dx\right) }^{l}\left\Vert
\Delta L_{Q}\right\Vert _{L^{\infty }\left( dx\right) } +\left\Vert a\left( x\right) \right\Vert _{L^{\infty }\left( dx\right)
}\left\Vert q\right\Vert _{L^{\infty }\left( dx\right) }^{l}\left\Vert
\Delta L_{Q^{\prime }}\right\Vert _{L^{2}\left( dx\right) } \\
\nonumber
&+\left\Vert xa\left( x\right) \right\Vert _{L^{2}\left( dx\right)
}\left\Vert q\right\Vert _{L^{\infty }\left( dx\right) }^{l}\left\Vert
\Delta q\right\Vert _{L^{\infty }\left( dx\right) } +t\left\Vert \partial _{x}a\left( x\right) \right\Vert _{L^{2}\left(
dx\right) }\left\Vert q\right\Vert _{L^{\infty }\left( dx\right)
}^{l}\left\Vert \Delta q\right\Vert _{L^{\infty }\left( dx\right) } \\
%&\leqslant {\norm{a(x)}{L^2}}\left\Vert
%\Delta r\right\Vert _{H^{1,1}} \Twomat{}{}{\frac{l-1}{2}+\frac{1}{2p}+\frac{1}{4}}{}\norm{L_Q}{L^\infty(dx)}+{\norm{a(x)}{L^\infty}}\left\Vert \Delta r\right\Vert _{H^{1,1}}\Twomat{}{}{\frac{l-1}{2}+\frac{1}{2p}+\frac{1
%}{4}}{}\norm{L_{Q'}}{L^\infty(dx)}\\
%&+{\norm{a(x)}{L^2}}
%&\left\Vert \Delta r\right\Vert _{H^{1,1}}\Twomat{}{}{\frac{l}{2}}{}+{\norm{a(x)}{L^\infty}}
%\left\Vert \Delta r\right\Vert _{H^{1,1}}\Twomat{}{}{\frac{l}{2}}{} \\
%&+{\norm{xa(x)}{L^2}}\left\Vert \Delta
%r\right\Vert _{H^{1,1}}\Twomat{}{}{\frac{l}{2}+\frac{1}{2p}+\frac{1}{4}}{}+{\norm{\partial_x a(x)}{L^2}}
%\left\Vert \Delta r\right\Vert _{H^{1,1}}\Twomat{}{}{\frac{l}{2}+\frac{1}{2p}-\frac{3}{4}}{} \\
\nonumber
&\leqslant \frac{C_{\Delta LG2}(\rho, \eta)}{(1+t)^{\frac{l}{2}+\frac{1}{2p}-\frac{3}{4}}}\left\Vert
\Delta r\right\Vert _{H^{1,1}}.
\end{align}
The difference in the $L^1$ norm can be estimated as
{\small\begin{align}
\label{est:DLG1}
\Vert \Delta \tilde{L}G\Vert _{L^{1}(dx)} &\leqslant  \left\Vert
a(x)\Delta \left( |q|^{l}\right) \tilde{L}\mathbf{Q}\right\Vert
_{L^{1}(dx)}+\left\Vert a(x)\Delta \left( |q|^{l-2}q^{2}\right) \tilde{L}%
\mathbf{Q}\right\Vert _{L^{1}(dx)}+\left\Vert \Delta \left( |q|^{l}q\right) 
\tilde{L}a(x)\right\Vert _{L^{1}(dx)}\\
\nonumber
&\quad +\left\Vert a(x)\left\vert q\right\vert
^{l}\Delta \tilde{L}\mathbf{Q}\right\Vert _{L^{1}(dx)}  \\
\nonumber
&\leqslant \left\Vert a\left( x\right) \right\Vert _{L^{1}\left( dx\right)
}\left\Vert q\right\Vert _{L^{\infty }\left( dx\right) }^{l-1}\left\Vert
\Delta q\right\Vert _{L^{\infty }\left( dx\right) }\left\Vert
L_{Q}\right\Vert _{L^{\infty }\left( dx\right) }+\left\Vert a\left( x\right) \right\Vert _{L^{2}\left( dx\right)
}\left\Vert q\right\Vert _{L^{\infty }\left( dx\right) }^{l-1}\left\Vert
\Delta q\right\Vert _{L^{\infty }\left( dx\right) }\left\Vert L_{Q^{\prime
}}\right\Vert _{L^{2}\left( dx\right) } \\
\nonumber
&+\left\Vert a\left( x\right) \right\Vert _{L^{1}\left( dx\right)
}\left\Vert q\right\Vert _{L^{\infty }\left( dx\right) }^{l}\left\Vert
\Delta L_{Q}\right\Vert _{L^{\infty }\left( dx\right) } +\left\Vert a\left( x\right) \right\Vert _{L^{2}\left( dx\right)
}\left\Vert q\right\Vert _{L^{\infty }\left( dx\right) }^{l}\left\Vert
\Delta L_{Q^{\prime }}\right\Vert _{L^{2}\left( dx\right) } \\
\nonumber
&+\left\Vert xa\left( x\right) \right\Vert _{L^{1}\left( dx\right)
}\left\Vert q\right\Vert _{L^{\infty }\left( dx\right) }^{l}\left\Vert
\Delta q\right\Vert _{L^{\infty }\left( dx\right) } +t\left\Vert \partial _{x}a\left( x\right) \right\Vert _{L^{1}\left(
dx\right) }\left\Vert q\right\Vert _{L^{\infty }\left( dx\right)
}^{l}\left\Vert \Delta q\right\Vert _{L^{\infty }\left( dx\right) } \\
%&\leqslant {\norm{a(x)}{L^1}}\left\Vert
%\Delta r\right\Vert _{H^{1,1}} \Twomat{}{}{\frac{l-1}{2}+\frac{1}{2p}+\frac{1}{4}}{}+{\norm{a(x)}{L^\infty}}\left\Vert \Delta r\right\Vert _{H^{1,1}}\Twomat{}{}{\frac{l-1}{2}+\frac{1}{2p}+\frac{1
%}{4}}{}\\
%&+{\norm{a(x)}{L^1}}
%\left\Vert \Delta r\right\Vert _{H^{1,1}}\Twomat{}{}{\frac{l}{2}}{}+{\norm{a(x)}{L^\infty}}
%\left\Vert \Delta r\right\Vert _{H^{1,1}}\Twomat{}{}{\frac{l}{2}}{} \\
%&+{\norm{xa(x)}{L^1}}\left\Vert \Delta
%r\right\Vert _{H^{1,1}}\Twomat{}{}{\frac{l}{2}+\frac{1}{2p}+\frac{1}{4}}{}+{\norm{\partial_x a(x)}{L^1}}
%\left\Vert \Delta r\right\Vert _{H^{1,1}}\Twomat{}{}{\frac{l}{2}+\frac{1}{2p}-\frac{3}{4}}{} \\
\nonumber
&\leqslant \frac{C_{\Delta LG1}(\rho, \eta)}{(1+t)^{ \frac{l}{2}+\frac{1}{2p}-\frac{3}{4}}}\left\Vert
\Delta r\right\Vert _{H^{1,1}}.
\end{align}}

and the calculation of $\|\tilde{L}G\|_{L^p}, \|\Delta \tilde{L}G\|_{L^p}, 1<p<\infty$ is  the same  as above.

\section*{Funding Declaration}The first author was partially supported by NSF grant DMS-2350301 and by Simons foundation MP-TSM-
00002258. The second author is partially supported by NSFC Grant No.12201605. The third author is partially supported by NSFC Grant No.12201605.

\section*{Declaration of Conflict of Interest}

The authors declare that they have no known competing  interests or personal relationships that could have appeared to influence the work reported in this paper.
\section*{data availability statement}
The authors declare that this manuscript has no associated data.
\bibliographystyle{amsplain}

\end{document}